\documentclass[12pt]{article}
\usepackage[nottoc,numbib]{tocbibind}
\usepackage[new]{old-arrows}

\usepackage[margin=1in]{geometry}
\usepackage{times,epsf,mathtools,xcolor,amssymb,amsmath, amsthm, graphicx,subcaption, comment, tikz-cd, appendix,booktabs,authblk}

\usepackage[
  matha,
  mathx,
]{mathabx}

\usepackage{enumitem}

\usepackage{hyperref}
\usepackage{xcolor}
\hypersetup{
    colorlinks,
    linkcolor={blue!80!black},
    citecolor={blue!50!black},
    urlcolor={blue!80!black}
}

\usepackage{cancel}
\usepackage{soul}
\usepackage{thmtools,cleveref}
\usepackage{framed}
\usepackage[normalem]{ulem}

\usepackage{thm-restate}

\newtheorem{thm}{Theorem}[section]
\newtheorem{lem}[thm]{Lemma}
\newtheorem{prop}[thm]{Proposition}
\newtheorem{cor}[thm]{Corollary}

\newtheorem{question}[thm]{Question}

\newtheorem*{lem*}{Lemma}
\newtheorem*{obs*}{Observation}
\newtheorem*{prop*}{Proposition}

\theoremstyle{definition}
\newtheorem{defn}[thm]{Definition}
\newtheorem{ex}[thm]{Example}

\theoremstyle{remark}
\newtheorem{rmk}[thm]{Remark}
\newtheorem*{rmk*}{Remark}

\newcommand{\R}{\mathbb{R}}
\newcommand{\Z}{\mathbb{Z}}

\newcommand{\rad}{\mathrm{rad}}
\newcommand{\conv}{\mathrm{conv}}
\newcommand{\cconv}{\overline{\mathrm{conv}}}

\newcommand{\PD}{\mathrm{PD}}

\newcommand{\wh}{\widehat} 
\newcommand{\wc}{\widecheck} 
\newcommand{\wt}{\widetilde}

\newcommand{\N}{\mathcal{N}}
\newcommand{\cN}{\overline{\mathcal{N}}}

\newcommand{\p}{\mathcal{P}}

\newcommand{\A}{\mathcal{A}}

\newcommand{\VR}{\mathcal{V}}
\newcommand{\C}{\mathcal{C}}

\newcommand{\M}{\mathcal{M}}
\newcommand{\V}{\mathcal{V}}
\newcommand{\T}{\mathcal{T}}

\newcommand{\h}{\mathrm{H}}
\newcommand{\X}{\mathcal{X}}

\newcommand{\Y}{\mathbf{Y}}

\newcommand{\E}{\mathrm{E}}
 \newcommand{\dd}{\mathbf{d}}

\DeclareMathOperator{\diam}{diam}
\DeclareMathOperator{\dist}{dist}
\DeclareMathOperator{\UW}{UW}
\DeclareMathOperator{\AW}{AW}
\DeclareMathOperator{\TW}{TW}
\DeclareMathOperator{\KW}{KW}

\DeclareMathOperator{\ts}{E}
\DeclareMathOperator{\spr}{spread}

\newcommand{\Hom}{\widetilde{\mathrm{H}}}
\newcommand{\CHom}{\widecheck{\mathrm{H}}}

\DeclareMathOperator{\uspr}{{\ddot{u}}-spread}

\DeclareMathOperator{\cdef}{cdef}
\DeclareMathOperator{\hcdef}{hcdef}

\DeclareMathOperator{\closure}{cl}

\begin{document}

\title{Geometric Bounds for Persistence}

\author[1]{Alexey Balitskiy\thanks{\href{mailto:alexey.balitskiy@uni.lu}{alexey.balitskiy@uni.lu}}}

\author[2]{Baris Coskunuzer
\thanks{\href{mailto:coskunuz@utdallas.edu}{coskunuz@utdallas.edu}}}

\author[3]{Facundo M\'emoli
 \thanks{\href{mailto:facundo.memoli@rutgers.edu}{facundo.memoli@rutgers.edu}}}
     
      \affil[1]{Department of Mathematics, University of Luxembourg}

      \affil[2]{Department of Mathematical Sciences, UT Dallas}

      \affil[3]{Department of Mathematics, Rutgers University}

\maketitle

\begin{abstract}

In this paper, we offer a new perspective on persistent homology by integrating key concepts from metric geometry. For a given compact subset $\X$ of a Banach space $\Y$, we analyze the topological features arising in the family $\N_\bullet(\X \subset \Y)$ of nested neighborhoods of $\X$ in $\Y$ and provide several geometric bounds on their persistence (lifespans).

We begin by examining the lifespans of these homology classes in terms of their filling radii in $\Y$, establishing connections between these lifespans and fundamental invariants in metric geometry, such as the Urysohn width. We then derive bounds on these lifespans by considering the $\ell^\infty$-principal components of $\X$, also known as Kolmogorov widths.

Additionally, we introduce and investigate the concept of  extinction time of a metric space $\X$: the critical threshold beyond which no homological features persist in any degree. We propose  methods for estimating the \v{C}ech and Vietoris--Rips extinction times of $\X$ by relating $\X$ to its convex hull and to its tight span, respectively. 
\end{abstract}

\tableofcontents

\section{Introduction}
\label{sec-intro}

Over the past decade, numerous approaches within Topological Data Analysis (TDA) have been developed to uncover patterns across a wide variety of data types. Among these, Persistent Homology (PH) has emerged as a cornerstone of TDA, providing a robust multiscale feature extraction framework. This progress has been driven by the development of efficient algorithmic procedures and effective software implementations for its computation (see \Cref{sec:PH}). Indeed, 
PH has gained considerable traction in diverse machine learning applications spanning fields such as:
\begin{itemize}[noitemsep,topsep=0pt]
\item Bioinformatics and Biomedicine~\cite{nicolau2011topology,chan2013topology,gameiro2015topological,cang2017topologynet,amezquita2020shape,bleher2021topology,skaf2022topological},
\item Finance~\cite{gidea2018topological,rudkin2023uncertainty}, 
\item Materials Science~\cite{hiraoka2016hierarchical,robins2016percolating,sorensen2022persistent,obayashi2022persistent,
lee2018high}, 
\item Neuroscience~\cite{curto2008cell,singh2008topological,dabaghian2012topological,reimann2017cliques}, and 
\item Network Analysis~\cite{hofer2017deep,hofer2019learning,hofer2020graph,carriere2020perslay,zhao2020persistence,aktas2019persistence}.
\end{itemize}

In this paper, we explore aspects of the Persistent Homology (PH) methodology through the lens of metric geometry. Our results offer new quantitative interpretations of the PH output, providing deeper insights into its structure and significance.

\subsection{Persistent Homology (in  a nutshell)}
We  recall the basic idea behind persistent homology. Let $\Delta_\bullet=\{\Delta_r\stackrel{\iota_{r,s}}{\longhookrightarrow} \Delta_s\}_{0< r\leq s}$ be a \emph{filtration}: a \emph{nested} family of topological spaces or  simplicial complexes, e.g. obtained via the Vietoris--Rips  filtration $\VR_\bullet(\X)$ or the \v{C}ech filtration $\C_\bullet(\X\subset \Y)$ induced from a compact metric space $\X$ (in the case of the \v{C}ech filtration, one typically assumes that $\X$ is a subset of a Banach space $\Y$). For a non-negative integer $k$, let $\omega$ be a nontrivial degree-$k$ reduced homology class appearing in the nested family $\Delta_\bullet$, that is, assume that $\omega \in \Hom_k(\Delta_r;\mathbb{F})$ for some $r>0$.\footnote{Here $\mathbb{F}$ is a fixed field.}   We in fact consider the degree-$k$ \emph{homological spectrum} of the filtration $\Delta_\bullet$, $\mathrm{Spec}_k(\Delta_\bullet)$ to be the collection  of all such non-zero homology classes (see \Cref{eq:spec}).   The \textit{birth time} $b_\omega$  of $\omega \in \Hom_k(\Delta_r;\mathbb{F})$ is the infimal $u> 0$ such that there exists $\omega_u\in \Hom_k(\Delta_u)$ with the property that $(\iota_{u,r})_*(\omega_u) = \omega$. Similarly, we define  the \textit{death time} $d_\omega$ of $\omega$ to be the supremal  $v\geq b_\omega$ such that $\omega$ does not become homologically trivial in $\Delta_v$, that is  $(\iota_{r,v})_*(\omega)\neq 0$. This is informally interpreted as indicating  that the nontrivial class $\omega$ is ``alive"  inside the interval $I_\omega = (b_\omega,d_\omega]$.\footnote{Whether the interval is left/right open/closed depends on semi-continuity conditions of $\Delta_\bullet$. See \Cref{defn:spectrum} for the case of neighborhood filtrations, the type of filtrations that we concetrate on in this paper.} We call the   quantity $d_\omega-b_\omega$ \textit{the lifespan} (or persistence) of the class  $\omega$. 

The notion of \emph{Persistent Homology} is closely related to but subtly differs from this process of recording birth and death times for individual homology classes described above.  The degree-$k$ persistent homology of $\Delta_\bullet$ is the directed system of vector spaces $\Hom_k(\Delta_\bullet;\mathbb{F})$. Under  suitable tameness assumptions on the family $\Delta_\bullet$, an up-to-isomorphism representation of this directed system  can be obtain via its \emph{persistence diagram}, a multiset of intervals $I$ on $\mathbb{R}_{>0}$ supporting certain linearly independent collection of nontrivial homology classes that are alive at all points in $I$.  See \Cref{sec:PH} for the precise definition of persistence diagrams (PD).

\medskip
In many applications, the lifespan of a topological feature $\omega$ is critically significant, as it is often interpreted as a measure of the ``size" or ``importance" of $\omega$. In practice, topological features with long lifespans—those that \emph{persist}—are typically considered to represent the primary shape characteristics of a dataset, while features with short lifespans are generally regarded as (topological) noise.\footnote{However, short lived, or even ephemeral, topological features also can carry useful information; see Usher and Zhang \cite{usher2016persistent}, Bubenik, Hull, Patel and Whittle~\cite{bubenik2020persistent}, and M\'emoli and Zhou \cite{memoli2022ephemeral}.} Therefore, \emph{determining} (or estimating) and \emph{interpreting} the lifespans of these topological features appearing in the persistence diagram has critical significance for applications such as the ones mentioned above.

\subsection{Connections with Metric Geometry and Main Results}
In this paper, we aim to give a \emph{geometric interpretation} of these lifespans by relating them to several notions from Metric Geometry.  We study the lifespans of individual homology classes $\omega$ 
appearing throughout the Vietoris--Rips and \v{C}ech filtrations by resorting to the notion of \emph{filling radius} and  to several notions of \emph{width}. See~\Cref{sec:related} for a discussion of the interplay between widths and filling radii in metric geometry. Through these notions of width, our results show that the lifespans of homology classes are controlled by the (geometric) size of their representatives in the filtration thus providing precise (geometric) interpretations of the significance of the features tracked by persistent homology.

\medskip
We first discuss implications of the absolute (Gromov's) and relative filling radius of a homology class (\Cref{defn:filling_radius}) in our setting. Then, we observe that the \v{C}ech lifespan of a homology class $\omega$ is equal to its relative filling radius in  ambient space (\Cref{sec:FR_background}). 
\
Next, we give several bounds for the lifespans of individual homology classes by resorting to the notions of Urysohn width, Alexandrov width, and Kolmogorov width. They measure in various ways how well a space can be approximated by a $k$-dimensional complex, and they are denoted by $\UW_k(\cdot), \AW_k(\cdot), \KW_k(\cdot)$, respectively. 
Since all these notions of width are monotonically non-increasing with respect to the dimension parameter (e.g., $\UW_k(\X)\geq \UW_{k+1}(\X)$),  any degree-$k$ estimate automatically applies to homology classes in higher degrees.

\begin{restatable*}[VR Lifespans via Urysohn Width]{cor}{corvruw}\label{cor:VR-UW} Let $\X$ be a compact metric space, and let $\omega \in \mathrm{Spec}_k(\VR_\bullet(\X))$, $k\ge 1$. Then, 
\[
d_\omega-b_\omega\leq \UW_{k-1}\big(\overline{\N}_{b_\omega}(\X\subset \ts(\X))\big).
\]
In particular, $$d_\omega-b_\omega\leq \UW_{k-1}(\ts(\X)).$$
\end{restatable*}

Here and throughout the paper, $\overline{\N}_r(\X\subset \mathcal{Z})$ denotes the closed $r$-neighborhood of $\X$ inside the metric space $\mathcal{Z}$ and $\ts(\X)$ denotes the \emph{tight span}  of $\X$ (\Cref{defn:tightspan}),  a canonically constructed metric space admitting an isometric embedding of $\X$ and enjoying properties reminiscent of (but stronger than) the ones satisfied by the convex hull. 

Notice that in a special case, if $\X$ is a closed $k$-manifold, and $\omega = [\X]$ is its fundamental class, since $b_\omega=0$ in that case, the result above implies that the Vietoris--Rips  lifespan of $\omega$ is bounded above by the Urysohn width of $\X$, i.e., $d_\omega\leq \UW_{k-1}(\X)$ (this particular bound goes back to Gromov; it follows from~\cite[Appendix~1, Example after Lemma~(B)]{gromov1983filling} combined with~\cite[Appendix~1, Proposition~(D)]{gromov1983filling}).

\medskip
Next, we give several bounds for \v{C}ech lifespans. The first one is via Alexandrov widths.

\begin{restatable*}[\v{C}ech Lifespans via Alexandrov Width]{cor}{coraw}\label{cor:cech-AW} Let $\X$ be compact subset of a Banach space $\Y$ and let $\omega\in\mathrm{Spec}_k(\C_\bullet(\X\subset \Y))$, $k\ge 1$. Then, 
\[
d_\omega-b_\omega\leq \AW_{k-1}(\cN_{b_\omega}(\X \subset \cconv(\X)) \subset\Y).
\]
In particular, $$d_\omega-b_\omega\leq \AW_{k-1}(\cconv(\X) \subset\Y).$$
\end{restatable*}

Here $\cconv(\X)\subset \Y$ denotes the closure of the convex hull of $\X\subset \Y$.

\medskip
Further, we introduce a new  notion of width, called {\em treewidth} (\Cref{defn:treewidth}), 
in order to obtain a finer estimate on \v{C}ech lifespans (\Cref{cor:cech-TW}). Notice that in the estimates above, for each  class $\omega$, in order to obtain an upper bound for its lifespan, the corresponding width needs to be calculated for the neighborhood $\N_{b_\omega}(\X \subset \cdot)$ of $\X$, not for $\X$ itself. In the following, we get rid of this dependency on the neighborhood. 

\medskip The notion of treewidth permits establishing a certain \emph{multiplicative} bound on lifespans.

\begin{restatable*}[\v{C}ech Lifespans via Treewidth]{cor}{cortw}\label{cor:cech-CTW}
     Let $\X$ be compact subset in a Banach space $\Y$ and  let $\omega\in\mathrm{Spec}_k(\C_\bullet(\X\subset \Y))$, with  birth time $b_\omega\geq 1$. Then, 
$$\dfrac{d_\omega}{b_\omega}\leq C+1+\TW^C_k(\X\subset \Y).$$
\end{restatable*}

Next, we consider the $\ell^\infty$-version of principal component analysis (PCA$_\infty$) for a compact subset $\X$ of a Banach space $\Y$ (e.g., a point cloud in $\R^N$). By using the estimates in previous sections, we relate the $(k+1)$-variance $\nu_{k+1}(\X\subset\Y)$ with the lifespans of classes $\omega$ appearing throughout the \v{C}ech filtration
(\Cref{sec:PCA}). 
We note that the variance $\nu_{k+1}(\X\subset \Y)$ is also known as the \textit{$k\textsuperscript{th}$ Kolmogorov width} of $\X$ in the approximation theory literature (\Cref{rmk:kolmogorov}).

\begin{restatable*}[\v{C}ech Lifespans via $\ell^\infty$-Variance]{cor}{corpca} \label{cor:PCA}
Let $\X$ be a compact subset of a Banach space $\Y$ and let  $\omega\in\mathrm{Spec}_k(\C_\bullet(\X\subset \Y))$, $k\ge 0$. Then, 
\[
d_\omega-b_\omega\leq \nu_{k+1}(\X\subset \Y) = \KW_k(\X\subset\Y).
\]
\end{restatable*}

Note that by monotonicity of widths (e.g., $\AW_k(\X\subset \Y)\geq \AW_{k+1}(\X\subset\Y))$, all these bounds apply to homology classes of degree higher than $k$ as well.

\medskip
While these lifespan bounds depend on the homology degree, we next give a general bound for \v{C}ech lifespans which is independent of degree. We achieve this by generalizing Katz's notion of spread \cite{katz1983filling} to the extrinsic setting and introducing the notion that we call \textit{\"uberspread} (\Cref{sec:spread}). \"Uberspread basically measures the Hausdorff distance from the space to the closest \emph{\"ubercontractible} space (a contractible space where all neighborhoods are also contractible, see \Cref{def:ubercontractible}) in an ambient space. With this notion, we generalize the existing VR-lifespan estimates via spread obtained by Lim, M\'emoli and Okutan  \cite{lim2020vietoris} to \v{C}ech-lifespans in any degree.

\begin{restatable*}[\v{C}ech Lifespans via \"Uberspread]{thm}{thmuberspread}\label{thm:uberspread}
    Let $\X$ be a compact subset of a Banach space $\Y$. 
Let $\omega\in\mathrm{Spec}_k(\C_\bullet(\X\subset \Y))$ for any $k\geq 0$. Then, $$d_\omega-b_\omega\leq 2\uspr(\X\subset \Y).$$
\end{restatable*}

While the results above are effective for bounding individual lifespans ($d_\omega-b_\omega$), we also attack a more general question: how to obtain a global bound for death times of homology classes across all degrees?  To do this we introduce a notion, called {\em extinction time},  representing the maximal threshold after which there is no nontrivial homology class in any degree $k\geq 0$. We bound both \v{C}ech extinction times $\wc{\xi}(\X\subset \Y)$ and Vietoris--Rips extinction times $\xi(\X)$ by relating $\X$ to their convex hulls and tight spans, respectively. In the case of \v{C}ech filtrations, we introduce a notion called {\em convexity deficiency}, $\cdef(\X\subset \Y)$, which is the Hausdorff distance of a space $\X$ to its convex hull in $\Y$ (\Cref{sec:Cechextinction}).

\begin{restatable*}[Bounding \v{C}ech Extinction]{thm}{thmcechextintion} \label{thm:extinct-cdef} Let $\X$ be a compact subset of a Banach space $\Y$. Then, 
$$\wc{\xi}(\X \subset \Y) \le \cdef(\X\subset\Y).$$ 
\end{restatable*}

In the VR case, we define an analogous notion called {\em hyperconvexity deficiency}, $\hcdef(\X)$, which is the Hausdorff distance between the tight span $\ts(\X)$ (\Cref{sec:VRextinction}) and the isometric copy of $\X$ inside of it. We then show that a result analogous to the above theorem is  also true in the VR-case~(\Cref{coro:abs-cont}).

\medskip
We highlight the bidirectional relationship between applied topology and metric geometry: on one hand, with the goal of improving their interpretability, we establish upper bounds for crucial quantities which originated in applied topology (e.g., lifespans of homology classes) using concepts from metric geometry; on the other hand, these results yield computational lower bounds or estimates for metric geometry notions inspired by persistent homology; see~\Cref{rmk:compute-fillrad}. This interplay underscores the  synergy between these fields, enabling insights that advance both domains.

\medskip
\noindent\textbf{Cores.}~Conceptually, our results establish a relationship between a given space $\X$ and another space $\Lambda_\X$, which functions as a \emph{core} for $\X$. This   represents the central thread weaving together the various parts of the paper. Specifically, estimates on the lifespans of homological features that arise as the radius of neighborhoods of $\X$ increases are derived from the distance between $\X$ and $\Lambda_\X$: 
\begin{itemize}
\item In \Cref{sec:width}, where we explore  various notions of widths, the role of the core $\Lambda_\X$ is, roughly speaking, assumed by a $k$-dimensional space closest to $\X$. In this context, the $k$-width can be interpreted as the distance between $\X$ and this approximate core, $\Lambda_\X$. In this section, we first recall the classical notions of Urysohn and Alexandrov width and then introduce a new variant which we call treewidth.

\item In \Cref{sec:spread}, after recalling the notion of spread, we introduce the notion of \"uberspread, where the core $\Lambda_\X$ is treated as an \"ubercontractible space, and the distance between $\X$ and $\Lambda_\X$ provides an upper bound on the lifespans in any degree. Notably, the condition imposed on the core in this context is the triviality of its homology groups, rather than any restriction on its dimension. 

\item In \Cref{sec:extinction}, where we analyze extinction times, we impose strong geometric conditions—such as convexity or hyperconvexity—on $\Lambda_\X$ to derive extinction bounds based on the distance between $\Lambda_\X$ and $\X$. This is done through the concepts of convexity deficiency and hyperconvexity deficiency, which we introduce therein.
\label{pg:cores}
\end{itemize}

\subsection{Related Work}
\label{sec:related} 
In this work, we aim to build a bridge between two seemingly disparate fields: applied algebraic topology and metric geometry. Both disciplines address a similar fundamental question regarding the quantification of ``shape" using distinct tools: 
\begin{center}{\em How to measure the size of a set/space/manifold?}
\end{center}

In applied algebraic topology, persistent homology is an effective tool for accomplishing this aim, and the \emph{lifespans} (or \emph{persistence}) of topological features induced by Vietoris--Rips (or \v{C}ech) filtrations are used as a measure of the size or importance of the corresponding topological features. In particular, Vietoris--Rips (or \v{C}ech) complexes were invented in order to transform a given metric space into a simplicial complex while maintaining its topological information, thus enabling an effective cohomology theory for metric spaces; see the papers by Vietoris~\cite{vietoris27}, Borsuk~\cite{borsuk1948imbedding} and Hausmann~\cite{Hausmann1995}. Numerous studies in the literature explore Vietoris–Rips complexes and Vietoris–Rips filtrations across various settings; see Latschev~\cite{latschev2001vietoris},  Chazal, Cohen-Steiner, de Silva, Guibas, M\'emoli, and Oudot~\cite{chazal2009gromov,chazal2014persistence}, Adamaszek, Adams, Frick, Gillespie, Lim, M\'emoli, Moy, Okutan, Reddy, and Wang~\cite{adamaszek2017vietoris, adamaszek2019vietoris,adams2021persistent,adams2022vietoris,lim2020vietoris,gillespie2024vietoris}, Attali, Lieutier and Salinas~\cite{ATTALI2013448}, Rieser, Bubenik and Milicevic~\cite{rieser2020vietoris,bubenik2024homotopy}, Turner~\cite{turner2019rips}, Virk \cite{virk20201,virk2022footprints,virk2021rips}, and Zaremsky~\cite{zaremsky2022bestvina}.

On the other hand, from the metric geometry side, estimating the size of a manifold has been a key problem for several decades. Gromov introduced and studied the notion of \textit{filling radius} in his seminal paper~\cite{gromov1983filling}, and several other notions of ``largeness'' in \cite{gromov1986large}. Before Gromov, certain relative, or extrinsic, versions of the filling radius were studied by Federer and Fleming~\cite{federer1960normal},  Michael and Simon~\cite{michael1973sobolev} and Bombieri and Simon~\cite{bombieri1983gehring} in geometric analysis in connection with the isoperimetric inequality. Gromov brought the filling radius to the realm of systolic geometry, and the study of scalar curvature~\cite{gromov1983filling}. Several other measures of size of a given manifold or metric space, nowadays known as \textit{widths}, were also studied and popularized by Gromov~\cite{gromov1983filling, gromov1988width}.

The interplay between filling radii, widths, and other metric invariants (including volume) has been an active research area since then. Katz determined the filling radius of spheres and other essential spaces~\cite{katz1983filling,katz1989diameter,katz1990filling,katz1991neighborhoods}. Several authors have studied the filling radii in comparison with other measures of largeness~\cite{cai1994gromov, brunnbauer2010large}. Guth proved some related conjectures of Gromov~\cite{guth2011volumes, guth2017volumes}. Sabourau, Nabutovsky, and Rotman related the filling radius with sweepouts of manifolds~\cite{sabourau2020one, nabutovsky2021sweepouts}. Bounds on filling radius in terms of Hausdorff content follow from very general isoperimetric estimates due to Liokumovich, Lishak, Nabutovsky, and Rotman~\cite{liokumovich2022filling}.

In recent years, several articles have explored connections between these two domains, addressing analogous problems with different methodologies. With this aim, Lim, M\'emoli and Okutan related the filling radius of a closed manifold to the interval corresponding to the fundamental class in the top VR persistence diagram~\cite{lim2020vietoris}.  As shown in \cite[Section 9.3.2]{lim2020vietoris}, the stability of persistence diagrams of Vietoris--Rips filtrations can be used to obtain  lower bounds for the Gromov--Hausdorff distance between spheres through considerations related to their filling radii. These  were shown not to be tight by Lim, M\'emoli and Smith~\cite{lim2021gromov}, and the polymath project ~\cite{adams2022gromov}  furthered this line work; see also Jeffs and Harrison~\cite{harrison2023quantitative} and Rodriguez-Mart\'{\i}n~\cite{martin2024gromov}.  Adams and Coskunuzer used a well-known quantity in metric geometry, Urysohn width, to estimate the lifespans in the persistence diagram of a given space~\cite{adams2022geometric}. In~\cite{virk2022footprints}, in the manifold setting, Virk studied the relation between  persistence diagrams for large degrees and lower dimensional features.

\medskip
In this paper, we aim to establish a direct connection between concepts from applied algebraic topology and metric geometry by linking various quantities used in both fields to measure the size of a metric space.  
In writing this paper, we have prioritized accessibility, aiming to bridge the gap between the two fields and foster greater collaboration and understanding.

 \subsubsection*{Ackowledgements}
 This work was supported by the following grants from the National Science Foundation (Grants \# DMS-1926686, DMS-2202584, DMS-2220613, DMS-2229417, DMS-1723003, IIS-1901360, CCF-1740761, and DMS-1547357), from the Simons Foundation (Grant \# 579977) and from the BSF (under grant \# 2020124). The project has received funding from the HORIZON EUROPE Research and Innovation programme under Marie Skłodowska-Curie grant agreement number 101107896. A major part of this work was done while the first-named author was a postdoctoral member at the Institute for Advanced Study, Princeton. We thank the anonymous reviewers for their feedback which helped enhance the presentation of this work.

\section{Background} 
\label{sec-background}

In this section, we provide an overview of the  concepts from applied algebraic topology and metric geometry that form the foundation of the paper. We give a summary of our notations in \Cref{notations} in the Appendix.

\subsection{Persistent Homology} \label{sec:PH} 

Persistent Homology (PH) is a methodology rooted in Applied Algebraic Topology that captures various structural characteristics of a given topological or metric space. Its development can be traced to the pioneering work of Frosini~\cite{frosini1990distance} and Robins~\cite{robins1999towards}, with its algorithmic framework later established by Edelsbrunner, Letscher, and Zomorodian~\cite{edelsbrunner2002topological}. Earlier manifestations  of persistent homology were retrospectively identified in the works of Morse~\cite{morse1930foundations}, Deheuvels~\cite{deheuvels1955topologie}, and Barannikov~\cite{barannikov1994framed}.

In the past two decades, PH has been employed as a powerful mathematical machinery for discovering patterns in data in applications within Machine Learning and Data Science. This advancement has been made possible by the development of efficient algorithms capable of computing PH in \emph{polynomial time}. Specifically, the total computational effort is a polynomial function of parameters related to the size of the input simplicial filtration and the maximum homology degree to be computed; see Edelsbrunner, Letscher and Zomorodian~\cite{edelsbrunner2002topological}, Harker, Mischaikow, Mrozek and Nanda~\cite{harker2014discrete,mischaikow2013morse}, and Bauer~\cite{bauer2021ripser}.

For more details on PH and its use in  various settings, see Carlsson~\cite{carlsson2009topology}, Edelsbrunner and Harer~\cite{edelsbrunner2010computational}, Chazal, de Silva, Glisse and Oudot ~\cite{chazal2016structure}, Ghrist~\cite{ghrist2018homological}, Rabad\'an and Blumberg \cite{rabadan2019topological},  Carlsson and Vejdemo-Johansson \cite{carlsson2021topological}, Joharinad and Jost~\cite{joharinad2023mathematical}, and Polterovich, Rosen, Samvelyan and Zhang~\cite{polterovich2020topological}.

\medskip

\noindent {\bf Neighborhood Notation.} 
Throughout the paper, we use both open and closed neighborhoods and adopt the following notation. Given a metric space $\mathcal{Z}$, a point $z\in \mathcal{Z}$ and $r>0$, by $B_r(z)$ we will denote the open ball of radius $r$ around $z$. When $\X$ is a subset of a metric space $\mathcal{Z}$, by  
$$\N_r(\X\subset \mathcal{Z}) := \bigcup_{x\in \X}B_r(x),$$ we will denote the open $r$-neighborhood of $\X$ in $\mathcal{Z}$ while $\overline{\N}_r(\X\subset \mathcal{Z})$ will denote the similarly defined closed $r$-neighborhood of $\X$ in $\mathcal{Z}$.

\subsubsection{Filtrations} \label{sec:filtration}

As noted in the previous section, utilizing the PH machinery requires a filtration—a nested family of topological spaces or abstract simplicial complexes—denoted by $\Delta_\bullet$. One of the most natural examples arises by considering nested neighborhoods of a subspace of a metric space, i.e., for $\X$ a subset of a metric space $\mathcal{Z}$,  the family $\{\N_r(\X\subset \mathcal{Z})\}_{r\geq 0}$ defines a filtration. Simplicial constructions are preferred in practical applications and the most common ones are Vietoris--Rips and \v{C}ech complexes and the respective filtrations they induce. While our study focuses primarily on these two types of filtrations, most of our results concerning \v{C}ech complexes naturally extend to alpha complexes; see \Cref{rmk:alpha}.

\begin{defn} [Vietoris--Rips Complexes] \label{defn:VR} Let $(\X,\dd_\X)$ be a compact metric space. For $r>0$, its Vietoris--Rips complex at scale $r$ is the abstract simplicial complex $\VR_r(\X)$ where a $k$-simplex $\sigma=[x_{i_0},x_{i_1},\dots,x_{i_k}]\in \VR_r(\X)$
if and only if $\dd_\X(x_{i_m},x_{i_n})< r$ for any $0\leq m,n\leq k$. 
\end{defn}

\begin{defn} [\v{C}ech Complexes] \label{defn:Cech} Let $\X$ be a compact subset of a metric space $\mathcal{Z}$. For $r>0$, the \v{C}ech complex at scale $r$ is the abstract simplicial complex $\C_r(\X \subset \mathcal{Z})$ where a $k$-simplex $\sigma=[x_{i_0},x_{i_1},\dots,x_{i_k}]\in \C_r(\X \subset \mathcal{Z})$ if and only if $\bigcap_{m=0}^k B_r(x_{i_m})\neq \emptyset$ in $\mathcal{Z}$. 
\end{defn}

In most scenarios $\mathcal{Z} = \Y$, a Banach space.

\medskip
Through the geometric realization functor, the nested families of simplicial complexes provided by the Vietoris--Rips and \v{C}ech complexes induce filtrations.

\begin{defn}[Filtration] \label{def:pers-fams}
A \emph{filtration} of a topological space  is a collection $\Delta_\bullet=\big(\Delta_r,\iota_{r,s} \big)_{0<r\leq s}$ such that for each $0<r\leq s$, $\Delta_r$ is a subset of the given topological space and $\iota_{r,s}: \Delta_r \hookrightarrow \Delta_s$ is the inclusion map. 
\end{defn}

When there is no risk of confusion we will simply say that $\Delta_\bullet$ is a filtration without mentioning the ambient topological space, with the understanding that it can be recovered as the colimit of $\Delta_\bullet$. In the remainder of the paper, we use the notation $\VR_\bullet(\X)$ and $\C_\bullet(\X\subset \Y)$ to denote the  filtrations  induced by the (geometric realizations of the) Vietoris--Rips and \v{C}ech complexes of $\X$, respectively.

\begin{ex}[Neighborhod Filtrations]
Another example of filtrations arising in geometric scenarios is the following. Let $\X\subset \mathcal{Z}$ be a nonempty compact subset of a metric space $(\mathcal{Z},\dd_\mathcal{Z})$. Then, one considers the filtration $\N_\bullet(\X\subset \mathcal{Z})$ given, for each $r>0$, by  the \emph{open} $r$-neighborhood $\N_r(\X\subset \mathcal{Z})$ of $\X$ in $\mathcal{Z}$. We will refer to any  filtration arising in that manner as a \emph{neighborhood filtration}. Of particular relevance to this paper will be the case when $\mathcal{Z}=\Y$, a Banach space. 
\end{ex}

Given the similarities in the definitions of Vietoris--Rips and \v{C}ech simplicial complexes, it is natural to expect certain relationships between these two types of complexes. By direct computation, it is straightforward to see that for any  compact $\X\subset \Y$, $$\C_r(\X\subset\Y)\subseteq\VR_{2r}(\X)\subseteq\C_{2r}(\X\subset\Y).$$

The Nerve Theorem directly relates the $r$-neighborhoods of $\X$ in $\Y$, $\N_r(\X\subset \Y)$, with the induced \v{C}ech simplicial complexes.

\begin{lem} [Nerve Theorem; Alexandrov~\cite{alexandroff1928allgemeinen} and Borsuk \cite{borsuk1948imbedding}] \label{lem:nerve} Let $\X$ be a compact subset of a Banach space $\Y$. For any $r>0$, $\N_r(\X\subset \Y)$ and $\C_r(\X\subset\Y)$ are homotopy equivalent to each other, i.e., $$\N_r(\X\subset \Y)\simeq \C_r(\X\subset\Y).$$ 
\end{lem}

There  are ``persistent", or ``functorial'', versions of this result; see  Bauer, Kerber, Roll and Rolle \cite{bauer2023unified} for an overview of different variants of the functorial nerve lemma. We will use the version below (see the discussion in \cite[Remark 4.4]{lim2020vietoris}) to relate the \v{C}ech filtration $\C_\bullet(\X\subset \Y)$ 
 and the Neighborhood filtration $\N_\bullet(\X\subset \Y)$.

\begin{thm}[Persistent Nerve Theorem {\cite[Proposition 4.5]{lim2020vietoris}}]\label{thm:cech-neigh}
There exist homotopy equivalences $\varphi_{s}:\C_{s}(\X\subset \Y)\rightarrow \N_s(\X\subset \Y)$ for each $s>0$ such that for each $t>s>0$ the following diagram commutes up to homotopy:
$$\begin{tikzcd}
\C_s(\X\subset \Y) \arrow[r, hook] \arrow[d, "\varphi_{s}"' , rightarrow]
& \C_t(\X\subset \Y) \arrow[d, "\varphi_{t}"]\\
\N_s(\X\subset \Y)\arrow[r,hook] & \N_t(\X\subset \Y)
\end{tikzcd}$$

\end{thm}

Note that Theorem \ref{thm:cech-neigh}  implies that the persistent homology of $\C_\bullet(\X\subset \Y)$ is isomorphic to that of $\N_\bullet(\X\subset \Y)$, a fact that we will repeatedly use in the sequel. 

\medskip Since any compact metric space $\X$ can be regarded as a subset  of $L^\infty(\X)$ (via its Kuratowski embedding, see  \Cref{def:kuratowski}), one obtains an analogous result providing a connection between the Vietoris--Rips filtration and the filtration $\N_\bullet(\X\subset L^\infty(\X))$  consisting of nested neighborhoods 
$$\big\{\N_s(\X\subset L^\infty(\X))\subseteq \N_t(\X \subset L^\infty(\X))\big\}_{0<s\leq t}$$ of $\X\subset L^\infty(\X)$.

\begin{cor}[{\cite[Theorem 4.1]{lim2020vietoris}}]\label{thm:vr-neigh}
There exist homotopy equivalences $\varphi_{s}:\VR_{2s}(\X)\rightarrow \N_s(\X\subset L^\infty(\X))$ for each $s>0$ such that for each $t>s>0$ the following diagram commutes up to homotopy:
$$\begin{tikzcd}
\VR_{2s}(\X) \arrow[r, hook] \arrow[d, "\varphi_{s}"' , rightarrow]
& \VR_{2t}(\X) \arrow[d, "\varphi_{t}"]\\
\N_s(\X\subset L^\infty(\X)) \arrow[r,hook] & \N_t(\X\subset L^\infty(\X)
\end{tikzcd}$$
\end{cor}

\begin{rmk}\label{rem:inj}
The result above remains valid if $L^\infty(\X)$ is substituted by any other \emph{injective} metric space admitting an isometric embedding of $\X$; see \Cref{sec:ts} for the definition and  \Cref{sec:VRextinction}, where we in particular utilize the tight span $\ts(\X)$ as one such choice. Also, the proof of \Cref{thm:vr-neigh} yields that $\C_\bullet(\X\subset L^\infty(\X))$ and $\VR_{2\bullet}(\X)$ are naturally homotopy equivalent; see~\cite[Section 4]{lim2020vietoris}.
\end{rmk}

\begin{rmk} [Alpha Complexes] \label{rmk:alpha} Note that while we only discuss VR and \v{C}ech filtered complexes in our paper, our results on \v{C}ech lifespans naturally apply to lifespans of homology classes induced by alpha complexes as $\A_r(\X\subset\Y)\simeq\C_r(\X\subset\Y)$ where $\A_r(\X\subset\Y)$ represents the alpha complex induced by $\X$ with distance threshold $r\geq 0$; see Edelsbrunner and Harer~\cite[III.4]{edelsbrunner2010computational}.
\end{rmk}

\subsubsection{Persistent Homology} \label{sec:PHbackground} Here, we recall  basic notions pertaining to persistent homology that are necessary for our setting. We will follow the presentation from~\cite[Section 2.1]{lim2020vietoris}.

\begin{defn}[Persistence Module]
A \emph{persistence module} $(V_r,\Phi_{r,s})_{0<r\leq s}$ over $\R_{>0}$ is a family of $\mathbb{F}$-vector spaces $V_r$ for some field $\mathbb{F}$ with morphisms $\Phi_{r,s}:V_r \to V_s$ for each $r \leq s$ such that 
\begin{itemize}
\item $\Phi_{r,r}=\mathrm{id}_{V_r}$,
\item $\Phi_{s,t}\circ \Phi_{r,s}=\Phi_{r,t}$ for each $r\leq s \leq t$.
\end{itemize}
\end{defn}
For conciseness we will denote by $V_\bullet$ the persistence module given by $(V_r,\Phi_{r,s})_{0<r\leq s}$. The morphisms $\Phi_{\bullet,\bullet}$ are referred to as the \emph{structure maps} of $V_\bullet$.  Note that a persistence module $V_\bullet$ can be regarded as a functor from the poset $(\mathbb{R}_{>0},\leq)$ to the category of vector spaces. 

\begin{defn}[Interval Persistence Module]
		Given an interval $I$ in $\R_{>0}$  and a field $\mathbb{F}$, the \emph{interval persistence module} induced by $I$ is the persistence module $\mathbb{F}_\bullet[I]$  is defined as follows: The vector space at $r$ is $\mathbb{F}$ if $r$ is in $I$ and zero otherwise. Given $r \leq s$, the morphism  corresponding to the pair $(r,s)$ is the identity if $r,s$ are both contained in $I$ and zero otherwise.  
	\end{defn}

 \begin{defn}[Barcode and Persistence Diagram]
For a given persistence module $V_\bullet$, if there is a multiset of intervals $(I_\lambda)_{\lambda\in\Lambda}$ such that $V_\bullet$ is isomorphic to $\bigoplus_{\lambda\in \Lambda} \mathbb{F}_\bullet[I_\lambda]$, then that multiset is referred to as a \emph{(persistence) barcode} associated to the persistence module $V_\bullet$. Persistence Modules admitting such a multiset of intervals are said to be \emph{interval decomposable}. The \emph{persistence diagram} of $V_\bullet$ is then given as the multiset of points $(b_\lambda,d_\lambda)\in \mathbb{R}^2$, where $b_\lambda$,  is the left endpoint of $I_\lambda$ and $d_\lambda$ is its right endpoint.\footnote{Not every persistence module is interval decomposable; see Crawley-Boevey \cite{crawley2015decomposition} for more details.}
	\end{defn}

In applied algebraic topology, many persistence modules arise as follows.

\begin{defn} [Persistent Homology of a Filtration] \label{defn:PH} For any $k\geq 0$, applying the $k$-dimensional \emph{reduced} homology functor (with coefficients in a field $\mathbb{F}$) to a filtration $\Delta_\bullet =\big(\Delta_r,\iota_{r,s} \big)_{0<r\leq s}$ produces the persistence (homology) module $\Hom_k(\Delta_\bullet;\mathbb{F})=(\Hom_k(\Delta_r;\mathbb{F}),\Phi^k_{r,s})_{0<r\leq s}$ where the morphisms $\Phi^k_{r,s}$ are those induced by $\iota_{r,s}$.\footnote{Note that we are using reduced homology in our definition in order to dispense with the usual infinite length bar at the level of degree zero persistent homology.}
\end{defn} 

\begin{framed}
    In what follows we will drop the field $\mathbb{F}$ from the notation since all of our results hold for an arbitrary choice of $\mathbb{F}$. 
\end{framed}

Under suitable assumptions, the persistence modules obtained from filtrations, as described above, are interval decomposable.  In particular, the persistence modules obtained from neighborhood filtrations of compact subsets of a Banach space are interval decomposable (so that they admit barcodes).

\begin{thm}[{\cite[Theorem 1]{lim2020vietoris}}]\label{theorem:pbarcode}
		Assume $\X$ is a compact subset of a Banach space $\Y$. Then there is a (unique) persistence barcode associated to  the persistence module $\Hom_k(\C_\bullet(\X\subset \Y))$. In particular, $\Hom_k(\V_\bullet(\X))$ admits (unique) persistence barcode.\footnote{In \cite[Theorem 1]{lim2020vietoris} the authors only contemplate the case of the $\Y$ being equal to $L^\infty(\X)$ for some compact metric space $\X$. However, the proof of their result directly applies to the setting in the statement.}  
	\end{thm}

We will henceforth use $\PD_k(\X)$ and $\wc{\PD}_k(\X\subset \Y)$ to respectively denote the persistence diagrams of the Vietoris--Rips and \v{C}ech filtrations of $\X$.

\subsubsection{Stability Theorems} \label{sec:stability} 
Persistence diagrams are an effective methodology for encoding  topological properties of the space $\X$ and its neighborhoods in $\Y$. Persistence diagrams are stable, as expressed by the following stability theorems. Informally, these state that if the shape and the size of two spaces are similar, then their persistence diagrams are close to each other. To give formal statements, let $\dd_b(\cdot,\cdot)$ denote the bottleneck distance between persistence diagrams; see Edelsbrunner and Harer  \cite{edelsbrunner2010computational}. Let $\dd_\h^\Y$ be the Hausdorff distance between two subsets of the same Banach space $\Y$, and let $\dd_{\mathrm{GH}}$ be the Gromov--Hausdorff distance between two metric spaces; see Burago, Burago and Ivanov \cite[Chapter 7]{burago2001course}.

\begin{lem} [{Stability Theorem---VR; \cite[Theorem 3.1]{chazal2009gromov} and \cite[Theorem~5.2]{chazal2014persistence}}] \label{lem:stability1} Let $\X$ and $\X'$ be two compact metric spaces. Then, 
$$\dd_b(\PD_k(\X),\PD_k(\X'))\leq 2\,\dd_{\mathrm{GH}}(\X,\X').$$ 
\end{lem}

\begin{lem} [{Stability Theorem---\v{C}ech; ~\cite[Theorem~5.6]{chazal2014persistence}}] \label{lem:stability2} Let $\X,\X'$ be two compact subsets of a Banach space $\Y$. Then, $$\dd_b(\wc{\PD}_k(\X\subset \Y),\wc{\PD}_k(\X'\subset \Y))\leq \dd^\Y_{\mathrm{H}}(\X,\X').$$ 
\end{lem}

Note that Lemma \ref{lem:stability2} implies Lemma \ref{lem:stability1}. Indeed, this was implicitly used in the proof of \cite[Theorem 3.1]{chazal2009gromov}. Notice that the coefficient $2$ does not appear in the second stability theorem.

\subsubsection{Homological Spectra}\label{sec:homological-spec}

As described in the introduction, the chief goal of our paper is to provide effective bounds for the lifetime of all homology classes that appear along a (geometric) filtration of a metric space. We will formulate and realize this goal in a setting that encompasses, but is more general than,  persistence diagrams. 

\begin{rmk}\label{rmk:spec}
Notice that it is not true that the only homology classes that show up across the filtration are those coming from the initial space. One well known example is that of the circle $S^1$ (with its geodesic distance) and the Vietoris--Rips filtration. Indeed, as shown by  Adamaszek and Adams in \cite{adamaszek2017vietoris}, whereas $\VR_t(S^1)$ has the homotopy type of $S^1$ for $t\in\left(0,\tfrac{2\pi}{3}\right]$, the homotopy type is that of $S^3$ as soon as $t\in\left(\tfrac{2\pi}{3},\tfrac{4\pi}{5}\right]$. In fact, they show that $\VR_\bullet(S^1)$ eventually attains the homotopy types of all odd-dimensional spheres.
\end{rmk}

\medskip
We recall some additional definitions and results from \cite{lim2020vietoris}.  

\begin{defn}\label{defn:spectrum}
For an integer $k\geq 0$, a given field $\mathbb{F}$, and a  filtration  $\Delta_\bullet=\big(\Delta_r,\iota_{r,s} \big)_{0<r\leq s}$, let 
\begin{equation}\label{eq:spec}
\mathrm{Spec}_k(\Delta_\bullet):=\bigcup_{r>0}\bigg(\Hom_k(\Delta_r;\mathbb{F})\backslash\{0\}\times\{r\}\bigg)
\end{equation} be the $k\textsuperscript{th}$ \emph{homological spectrum} of $\Delta_\bullet$ (with coefficients in $\mathbb{F})$. 
\end{defn}
Now, fix an arbitrary $(\omega,s)\in\mathrm{Spec}_k(\Delta_\bullet)$. Then, let
\begin{align}\label{eq:lifespan}
    b_{(\omega,s)}&:=\inf\{r>0:r\leq s\text{ and }\exists\text{ nonzero }\omega_r\in\Hom_k(\Delta_r)\text{ such that }(\iota_{r,s})_\ast(\omega_r)=\omega\},\\
    d_{(\omega,s)}&:=\sup\{t>0:t\geq s\text{ and }\exists\text{ nonzero }\omega_t\in\Hom_k(\Delta_t)\text{ such that }(\iota_{s,t})_\ast(\omega)=\omega_t\}\\
   & \,\,= \sup\{t\geq s: (\iota_{s,t})_\ast(\omega)\neq 0\}. 
\end{align}

Whenever $\Delta_\bullet$ is a neighborhood filtration $\N_\bullet(\X\subset\Y)$,  as in \cite[Theorem 8]{lim2020vietoris}, one has that $b_{(\omega,s)}<s\leq d_{(\omega,s)}$.\footnote{In general, the type of intervals (open-open, closed-open, etc) one obtains depends on whether the filtration is defined via open or closed neighborhoods. Note that we've defined neighborhood filtrations via \emph{open} neighborhoods.} 
Let 
$$I_{(\omega,s)}:=
(b_{(\omega,s)},d_{(\omega,s)}].$$ Informally, the interval $I_{(\omega,s)}$ encodes the maximal region around $s \in \mathbb{R}_{>0}$ inside which the class $\omega$  is ``alive".\footnote{
Note that, since we are using reduced homology, there is no degree 0 class $\omega$ such that $d_{(\omega,s)} =\infty$.}

\begin{defn}\label{def:lifespan}
The value $b_{(\omega,s)}$ is referred to as the \emph{birth time} of $\omega$ whereas $d_{(\omega,s)}$  is the \emph{death time} of $\omega$. The value $d_{(\omega,s)}-b_{(\omega,s)}$ will be referred to as the \emph{lifespan} of $\omega.$   
\end{defn}

\begin{rmk}[A Caveat]\label{rem:closed-open-cech} We focus on the special case of neighborhood filtrations. The birth time $b_{(\omega,s)}$ was defined as the infimum of all times $r\le s$ when a ``predecessor'' of $\omega$ exists in $\Hom_k(\N_{r}(\X\subset \Y))$, and it is natural to ask whether there exists a homology class  supported on 
\[
\bigcap\limits_{b_{(\omega,s)} < r \le s} \N_{r}(\X\subset \Y) = \overline{\N}_{b_{(\omega,s)}}(\X\subset \Y)
\]
that is also homologous to $\omega$ in $\N_{s}(\X\subset \Y)$. It turns out this is not the case even for the neighborhood filtration of a compact set in $\mathbb{R}^2$. Namely, there exists a compact set $\X \subset \mathbb{R}^2$, known as the \emph{Warsaw circle} or \emph{closed topologist's sine curve},  satisfying the following counter-intuitive property: its first singular homology is zero, but every open neighborhood of it contains a homologically non-trivial circle; see Borsuk~\cite{borsuk1975theory}. Every two of those circles are homologous to each other (within the union of the two neighborhoods). But these circles do not converge, as we shrink the neighborhood, to a non-trivial homology class of $\X$, because $\Hom_1(\X) = 0$. 
The natural way to treat this ``limit circle'' is to consider  \v{C}ech homology instead of  ordinary singular homology. The \v{C}ech homology $\CHom_1(\X)$ is non-trivial, and contains the ``limit circle''.
In general, a predecessor of the class $\omega$ naturally lives in the \v{C}ech homology $\CHom_k(\overline{\N}_{b_{(\omega,s)}}(\X\subset \Y))$. Nonetheless, we consistently use singular homology throughout the paper in order to avoid overly technical details.

\end{rmk}

\begin{framed}
To ease the notational burden, we will often drop the parameter $s$ and will use the more succinct notation  $b_\omega$, $d_\omega$, $I_{\omega}$, etc.
\end{framed}

\begin{rmk}\label{rmk:not-all-omegas-in-spec}
A priori, one would expect  the collection of all intervals $$\big\{I_{\omega};\,\omega\in \mathrm{Spec}_k(\Delta_\bullet)\big\}$$ to be closely related to the $k\textsuperscript{th}$ persistence diagram of $\Delta_\bullet$. Whereas Proposition \ref{prop:relation} below establishes  a sense in which this is the case, it is not always true that all such intervals appear in the interval decomposition of $\Hom_k(\Delta_\bullet)$ (whenever it exists). An example showing this discrepancy in the case of Vietoris--Rips filtrations can be found in \cite[Example 9.16]{lim2020vietoris}.  
\end{rmk}

\begin{prop}[{\cite[Proposition 9.2]{lim2020vietoris-arxiv}}]\label{prop:relation}
Let $\Delta_\bullet=\big(\Delta_r,\iota_{r,s} \big)_{0<r\leq s}$ be a neighborhood filtration and let $k\geq 1$ be an integer. Then, for all $r < s$, the multiplicity of the interval $(r,s]$ in the barcode of $\h_k(\Delta_\bullet)$ is equal to
\begin{equation*}
  \max\left\{ m\in\Z_{\geq 0} \Bigg| \begin{array}{l}
    \exists\text{ \emph{linearly independent vectors} }\omega_1,\dots,\omega_m\in \h_k(\Delta_s)\text{ \emph{s.t.} }
    I_{(\omega_i,s)}=(r,s] \,   \forall i\\ \text{ \emph{and no nonzero linear combination of  these vectors belongs to} }\mathrm{Im}((i_{r,s})_*)
  \end{array}\right\}.
\end{equation*}

\end{prop}

This proposition indicates that for each interval $I = (r,s]$ in the barcode of $\Delta_\bullet$ there is a finite linearly independent collection $\omega_1,\ldots, \omega_m \in \h_k(\Delta_s)$ satisfying the conditions above such that $I = I_{(\omega_i,s)}$ for all $i$. One calls any such $\omega_i$ a \emph{representative} of the interval $I$.\footnote{In \cite[Proposition 9.2]{lim2020vietoris} the authors consider the case of the VR filtration of a totally bounded metric space. The same proof applies to the more general statement given above.} See \cite[page 42]{lim2020vietoris} for an example demonstrating the role of the condition that no nonzero linear combination of these vectors 
belongs to $ \mathrm{Im}((\iota_{r,s})_*)$.

\begin{rmk}\label{rmk:comparison}
The above proposition implies that if $I$ is an interval in the barcode of a neighborhood filtration, then $I = I_{\omega_i}$ so that, in particular, its length is equal to that of $I_{\omega_i}$, where $\omega_i$ is as in the statement. Therefore, and as we will do in the rest of the paper, if we have an upper bound for the length of all intervals $I_\omega$ where $\omega\in\mathrm{Spec}_k(\Delta_\bullet)$ then we will automatically have an upper bound for the length of every interval in the barcode of $\Delta_\bullet$.
\end{rmk}

\subsection{Geometry of $L^\infty(\X)$ and Tight Spans}\label{sec:ts}

While the discussion in this paper applies to subsets of any Banach space, special attention is paid to the important case of $L^\infty$ spaces. There are two main reasons for that. The first one is that if we start with a compact metric space that is a priori not a isometrically embedded into a Banach space, there is a nice way of placing it inside $L^\infty(\X)$, the space of bounded functions on $\X$ with the supremum norm.

\begin{defn}[Kuratowski Embedding] \label{def:kuratowski}
For a compact metric space $(\X,\dd_\X)$, the map $\kappa:\X \to L^\infty(\X)$, defined as $x \mapsto \dd_\X(x,\cdot)$, is a distance-preserving embedding, and it is called the \emph{Kuratowski embedding}. 
\end{defn}

The second reason is that $L^\infty$ spaces enjoy the \emph{hyperconvexity} property: If several balls have non-empty pairwise intersection then they share a point in common. 

\begin{defn}[Hyperconvex Space] \label{def:hyper}
A metric space $(\mathcal{E},\dd_\mathcal{E})$ is called \emph{hyperconvex} if for every family $(x_i,r_i)_{i \in I}$ of $x_i$ in $\mathcal{E}$ and $r_i \geq 0$ such that $\dd_\mathcal{E}(x_i,x_j) \leq r_i + r_j$ for every $i,j$ in $I$, there exists a point $x\in \mathcal{E}$ such that $\dd_\mathcal{E}(x_i,x) \leq r_i$ for every $i$ in $I$.
\end{defn}

The hyperconvexity of $L^\infty(\X)$ implies that the \v{C}ech and VR filtrations of a subset of $L^\infty(\X)$ coincide (up to a factor of two in the filtration index). This gives a way to study VR-lifespans of homology of $\X$ by immersing it in $L^\infty(\X)$ and then using the tools applicable for \v{C}ech-lifespans. Note that this idea has been successfully used in the context of VR-persistence, cf.~\Cref{prop:fillrad}. 

Hyperconvexity implies certain universal properties of $L^\infty$ spaces, and these properties will be implicitly used below in relation to Urysohn width (see \Cref{rmk:fillrad-width}), and tight spans (a.k.a. injective hulls). We briefly discuss the latter now and  we refer the reader to Lang's survey~\cite{lang2013injective} for more information.

\begin{defn}[Injective Metric Space]\label{def:Injme}
A metric space $\mathcal{E}$ is called \emph{injective} if for every $1$-Lipschitz map $\phi: \X \to \mathcal{E}$ and distance preserving embedding of $\X$ into $\wt{\X}$, there exists a $1$-Lipschitz map $\wt{\phi}: \wt{\X} \to \mathcal{E}$ extending $\phi$:
$$\begin{tikzcd}
\X \arrow[r, hook] \arrow[dr, "\phi"' , rightarrow]
& \wt{\X} \arrow[d, "\wt{\phi}",dotted]\\
& \mathcal{E}
\end{tikzcd}$$
\end{defn}

It turns out that injectivity coincides with hyperconvexity.

\begin{prop}\label{prop:hypinj}
A metric space is injective if and only if it is hyperconvex.
\end{prop}

The proof of this proposition can be found in Aronszajn and  Panitchpakdi~\cite{aronszajn1956extension} and in~\cite[Proposition 2.3]{lang2013injective}.

\medskip
It is known that $L^\infty(\X)$ is injective~\cite{lang2013injective}. However, there exists a more efficient injective space containing $\X$.
\begin{defn}[Tight Span] \label{defn:tightspan}
The \emph{tight span} $\ts(\X)$ of a compact metric space $\X$ is the minimal injective metric space admitting an isometric embedding of $\X$. Minimality here means that any other injective metric space admitting an isometric embedding of $\X$ contains an isometric copy of~$\ts(\X)$. 
\end{defn}

Tight spans are sometimes called \emph{injective hulls} or \emph{hyperconvex hulls}. The notions of injectivity and hyperconvexity  were  first proposed by Aronszajn and Panitchpakdi~\cite{aronszajn1956extension}.  Isbell~\cite{isbell1964six} first identified the notion of tight span (although the author used the term \emph{injective envelope}). Additional contributions were made by Dress~\cite{dress1984trees} and Lang~\cite{lang2013injective}; see Chepoi~\cite{chepoi1997atx} for a historical account. 

\begin{prop}[{Properties of the Tight Span \cite{lang2013injective}}] \label{prop:ts}
The tight span $\ts(\X)$ of a compact metric space $\X$ exists and satisfies: 
\begin{enumerate}
\item $\ts(\X)$ is compact.
\item $\ts(\X)$ is contractible.
\item $\ts(\X)$ is isometric to $\X$ for any metric tree $\X$. 
\item $\diam(\ts(\X))=\diam(\X)$. 
\end{enumerate}
\end{prop}

One particular realization of the tight span $\ts(\X)$ of $\X$ as a subset of $L^\infty(\X)$ is given as follows \cite[Section 3]{lang2013injective}):$$\ts(\X):=\{f\in\Delta(\X):\text{if }g\in\Delta(\X)\text{ and }g\leq f,\text{ then }g=f\,(\text{i.e., }f\text{ is minimal})\} 
$$
where 
$$\Delta(\X):=\{f\in L^\infty(\X):f(x)+f(x')\geq \dd_\X(x,x')\text{ for all }x,x'\in \X\}.$$

Note that from the realization of the tight span recalled above, for any $f\in \ts(\X)$ and any $x\in \X$, it holds that 
\begin{equation}\label{eq:char}f(x) = \max_{x'\in\X}\big(\dd_\X(x,x')-f(x')\big)
=\big\|\dd_\X(x,\cdot)-f\big\|_\infty.\end{equation}

\begin{rmk}\label{rmk:tightretraction}
Not only can the tight span $\ts(\X)$ be regarded as a subset of $L^\infty(\X)$ but also, directly from the fact that it is an injective metric space,  there is a $1$-Lipschitz retraction of $L^\infty(\X)$ to $\ts(\X)$; 
see \cite[Proposition 2.2]{lang2013injective}. 
\end{rmk}

\subsection{Filling Radius} \label{sec:FR_background} The filling radius is a key notion in metric geometry introduced by Gromov as a measure of largeness of a given closed manifold~\cite{gromov1983filling}. To \emph{fill} a manifold $\M$ of dimension $n$, treated as a singular $n$-cycle, means to find an $(n+1)$-dimensional singular chain $\mathcal{D}$ with boundary $\M$, i.e. $\partial \mathcal{D}=\M$. In this case we will also be saying that the cycle $\M$ \emph{bounds}. The ambient space in which filling happens, as well as the coefficients of singular homology, should be specified, as we discuss below.

\begin{defn} [{Gromov's Filling Radius, \cite{gromov1983filling}}] \label{defn:filling_radius} The filling radius $\rho(\M)$ of a closed $n$-dimensional Riemannian manifold $\M$ is the
infimal number $R > 0$ such that $\M$ can be filled inside of the $R$-neighborhood of its Kuratowski image in $L^\infty(\M)$. 
\end{defn}

\begin{rmk}[Coefficients]
This definition makes sense with any homology coefficients. Some common choices include $\Z$-coefficients if $\M$ is oriented, $\Z_2$-coefficients if $\M$ is not oriented, and $\mathbb{Q}$-coefficients in some contexts where torsion is a problem. However, the usual persistence homology is well-defined only over fields,\footnote{See, however, \cite{patel2018generalized}.} so in the rest of the paper we implicitly assume that an arbitrary choice of a field is made (e.g., $\mathbb{Z}_2$), and that all filling radii and all persistence features are considered over this field.  All of our results hold for any field and, for this reason, the field is omitted from the notation.
\end{rmk}

Before Gromov, a different type of filling radius notion was discussed for submanifolds (or more generally, cycles) in $\R^N$ in geometric analysis, especially in the context of the isoperimetric problem~\cite{federer1960normal, michael1973sobolev,bombieri1983gehring}. We summarize both types of filling radii in the following definition, adapting it to the context of persistence.

\begin{defn} [Relative and Absolute Filling Radii of a Homology Class] 
\label{defn:homology_filling_radius} 
Let $(\X,\dd_\X)$ be a metric space, and let $\omega \in \Hom_k(\X)$ be a non-trivial reduced singular homology class. 
\begin{enumerate}
 \item Assume additionally that $\X$ is a subset of a Banach space $\Y$, so that the metric $\dd_{\X}$ agrees with the one inherited from $\Y$ (that is, the embedding $\X \hookrightarrow \Y$ is distance-preserving). The \emph{relative filling radius} of $\omega$ is the infimal number $r$ such that the image of $\omega$ under the induced map $\Hom_k(\X) \to \Hom_k(\N_r(\X \subset \Y))$ is trivial. In other words, it is the infimal size of a neighborhood in $\Y$ where some cycle representing $\omega$ bounds a $(k+1)$-chain. We will use the following notation for the relative filling radius: $\rho(\omega; \X \subset \Y)$.
 \item In case when no ambient space $\Y$ is specified, one can take $\Y = L^\infty(\X)$, and measure the relative filling radius of the Kuratowski image $\kappa(\X) \subset L^\infty(\X)$. This way one obtains the \emph{(absolute) filling radius} of $\omega$: $\rho(\omega; \X) = \rho(\kappa_*(\omega); \kappa(\X) \subset L^\infty(\X))$; cf. \cite[Definition~24]{lim2020vietoris}. 
\end{enumerate}
\end{defn}

The word ``absolute'', which is usually omitted, is justified by the fact that $\rho(\omega; \X)$ equals the infimum of $\rho(\iota_*(\omega); \iota(\X) \subset \Y)$ over all distance-preserving embeddings $\iota: \X \hookrightarrow \Y$~\cite[page 8]{gromov1983filling}; we will refer to this as the \emph{universal property}. Note that when $\X = \M$ is a closed Riemannian manifold, Gromov's filling radius (\Cref{defn:filling_radius}) over $\mathbb{Z}_2$ is the same as the absolute filling radius of the $\mathbb{Z}_2$-fundamental class $[\M]$.

\medskip The relationship between the VR filtration of a compact metric space and the absolute filling radii $\rho(\omega;\X)$ is studied in \cite[Section 9.3]{lim2020vietoris}.  The following proposition applies in the general context of  absolute neighborhood retracts (ANRs), which includes Riemmanian manifolds, metric graphs and other commonly appearing metric spaces.\footnote{Recall that an  ANR is any metric space $\X$ with the property that whenever it is embedded into another metric space $\mathcal{Z}$ through a homeomorphism $h:\X\to\mathcal{Z}$, then there is an open neighborhood $U$ of $h(\X)$ such that $h(\X)$ is a retract of $U$; see Borsuk~\cite{borsuk1932klasse} and Hu~\cite{hu1965theory}. It is known that every compact
(topologically) finite-dimensional locally contractible metric space is an ANR. Thus, all Riemannian manifolds are ANRs.}

\begin{prop}[{\cite[Propositions 9.28 and 9.46]{lim2020vietoris}}] \label{prop:fillrad}
Let $\X$ be a compact ANR metric space. Then, for any integer $k\geq 1$ and any nonzero $\omega\in \Hom_k(\X)$ we have:
\begin{itemize}
\item $\rho(\omega;\X)>0$;
\item the interval $(0,2\rho(\omega;\X)]$ appears in the degree-$k$ barcode of $\VR_\bullet(\X).$
\end{itemize}
Additionally, if $\X$ is a (closed and connected) Riemannian manifold, then there are no other intervals with left-endpoint (birth time) equal to zero in the degree-$k$ barcode of $\VR_\bullet(\X)$.\footnote{In particular, if $\X$ is an $n$-dimensional Riemannian manifold, then $(0,2\rho(\X)]$ is the unique interval with left endpoint zero in the degree-$n$ barcode.}
Here both the filling radius and persistent homology can be computed with coefficients in an arbitrary field when $\M$ is orientable,  and with coefficients in $\Z_2$ when $\M$ is not orientable. 
\end{prop}

\begin{rmk} [Relative Filling Radius vs. Absolute Filling Radius] \label{rmk:ellipsoid1} To give an idea about the difference between relative filling radius in $\Y$ and absolute filling radius, here we give a toy example of a flying saucer in $\Y=\R^3$ (with Euclidean metric). Let $\E$ be the ellipsoid in $\R^3$ given by $\E=\{(x,y,z)\mid \frac{x^2}{100}+\frac{y^2}{100}+\frac{z^2}{1}=1\}$, and let $[\E]$ be its fundamental class in $\h_2(\E)$. There are two different ways to treat $\E$ as a metric space, resulting in the different values of its filling radius. 
\begin{enumerate}
 \item One way is to consider the Riemannian metric $g$ on $\E$ induced by the Euclidean metric of $\R^3$. The Riemannian surface $\E_g$ thus obtained has the \textit{absolute filling radius} of about $10$: $\rho([\E]; \E_g) \sim 10$ (the exact computation is tricky).

 \item The other way is to borrow the extrinsic distance function from $\R^3$. This does not make $\E$ a Riemannian manifold, but rather just a compact metric space embedded in $\R^3$ in a distance-preserving way. The corresponding \textit{relative filling radius} equals $1$: $\rho([\E]; \E \subset \R^3) = 1$.
 
\end{enumerate}

In this example we have $\rho([\E]; \E_g) > \rho([\E]; \E \subset \R^3)$, which might seem to contradict the note above saying that $\rho(\omega; \X) \le \rho(\omega; \X \subset \Y)$ for distance-preserving embeddings $\X \subset \Y$. There is no contradiction here: even though the embedding $\E_g \subset \R^3$ is a Riemannian isometry, it is not distance-preserving. Indeed, the extrinsic distance in $\R^3$ between two points of $\E$ is smaller than the intrinsic distance between them inside $\E$ (the length of the shortest path in $\E$). If one computes the filling radius of $E$ with the extrinsic metric of $\R^3$, it will be at most $1$.

In the following sections, we will see that Gromov's filling radius of a Riemannian manifold isometrically embedded in an ambient space $\Y$ can highly overestimate the lifespan of a topological feature in its \v{C}ech filtration and that the relative filling radius is better adapted to this context. 
 \end{rmk}

\begin{rmk}[Comment about Estimation of Filling Invariants]\label{rmk:compute-fillrad}
The type of connections between persistent homology and metric geometry that we explore in this paper have the potential of permitting the estimation of quantities such as the filling radii $\rho(\omega;\X)$, $\omega\in\Hom_k(\X)$, via the polynomial time algorithms that have been developed for computing persistent homology (see \Cref{sec:PH}). Indeed, such an estimate would be obtained via \Cref{prop:fillrad} and the stability of PH (\Cref{lem:stability1}) through computing the VR-barcodes of a carefully chosen $\epsilon$-net for $\X$, for some $\epsilon >0$.
\end{rmk}

\smallskip
\noindent {\bf Filling Radii \& Lifespans.} 
The notion of \textit{lifespan} of a homology class (\Cref{def:lifespan}) is directly related to the filling radius as follows. Let $\X$ be a compact subset of a Banach space $\Y$. Let $(b_\omega,d_\omega)$ be the homological birth and death times of a degree-$k$ homology class $\omega$ present at time $s$ in the neighborhood filtration $\N_\bullet(\X\subset \Y)$. Directly from the definition of the relative filling radius, we obtain
\[
d_\omega - s = \rho(\omega; \N_s(\X\subset\Y) \subset \Y).
\]
A bit more generally, if $\omega_r \in \Hom_k(\N_{r}(\X\subset \Y))$ is a predecessor of $\omega$ (that is, it is mapped to $\omega$ by the map in homology induced by the inclusion), then
\[
d_\omega - r = \rho(\omega_r; \N_r(\X\subset\Y) \subset \Y).
\]
Letting $r \to b_\omega$, we obtain a formula for the lifespan of $\omega$ in terms of the filling radii of the predecessors of $\omega$:
\begin{align*}\label{eq:lifespan-filling-radius}
d_\omega - b_\omega &= \sup_{b_\omega < r \le s} \rho(\omega_{r}; \N_{r}(\X\subset\Y) \subset \Y) \\
&= \lim_{r \to b_\omega + 0} \rho(\omega_{r}; \N_{r}(\X\subset\Y) \subset \Y). \tag{$\star$}
\end{align*}

Note that both sides of this formula depend on homology coefficients lying in the same field, which can be arbitrary.

\begin{rmk}\label{rmk:cech-pred} As clarified in \Cref{rem:closed-open-cech}, this lifespan cannot be written in term of the filling radius of a homology class of $\overline{\N}_{b_\omega}(\X \subset \Y)$. We cannot expect (a predecessor of) $\omega$ to be present in the homology of $\overline{\N}_{b_\omega}(\X\subset \Y)$, unless we work with \v{C}ech homology, and redefine the filling radius correspondingly.
 \end{rmk}

Formula~\eqref{eq:lifespan-filling-radius} was given for the neighborhood filtration. In view of the functorial nerve theorem (\Cref{thm:cech-neigh}, \Cref{thm:vr-neigh}), the formula specializes to two important cases:
\begin{itemize}
    \item For any $\omega\in\mathrm{Spec}_k(\C_\bullet(\X\subset \Y))$,
\[ 
d_\omega - b_\omega = \lim_{r \to b_\omega + 0} \rho(\omega_{r}; \N_{r}(\X\subset\Y) \subset \Y). 
\]
    \item For any $\omega\in\mathrm{Spec}_k(\VR_\bullet(\X))$,
\[ 
d_\omega - b_\omega = 2\lim_{r \to b_\omega + 0} \rho\big(\omega_{r}; \N_{r}(\kappa(\X)\subset L^\infty(\X)) \subset L^\infty(\X)\big). 
\]
\end{itemize}

\begin{rmk}\label{rmk:vr-filling-redundancy}
For the particular case of VR-lifespans, i.e. when $\Y = L^\infty(\X)$, it is important to understand the behavior or the filling radii 
$\rho(\omega_{r}; \N_{r}(\kappa(\X) \subset L^\infty(\X)) \subset L^\infty(\X))$. We make a comment that this cumbersome notation is a bit redundant; it turns out that
\[
\rho\big(\omega_{r}; \N_{r}(\kappa(\X) \subset L^\infty(\X)) \subset L^\infty(\X)\big) = \rho\big(\omega_{r}; \N_{r}(\kappa(\X) \subset L^\infty(\X))\big).
\]
This is not obvious by default, since the absolute filling radius in the right-hand side should be computed in the space $L^{\infty}\big(\N_{r}(\kappa(\X) \subset L^\infty(\X))\big)$. Nevertheless, this equality holds true, as it is explained in \Cref{sec:VR_Urysohn} (see~\Cref{lem:fillrad_redundancy}).

\end{rmk}

\medskip

\noindent {\bf Auxiliary definitions.} Before concluding the background section, we define two versions of the radius of a set $\X$, which will be used in the remainder of the paper. In the first one, there is no reference to an ambient space, and the center is in the set $\X$. In the second one (circumradius), the radius of $\X$ is computed in an ambient space, and the center may not be in $\X$.

\begin{defn} [Radius] \label{def:radius}
Let $(\X, \dd_\X)$ be a compact metric space.
\begin{enumerate}
 \item The \emph{radius} of $\X$ is
 \[
 \rad(\X) := \inf_{x_0\in \X} \sup_{x\in \X} \dd_\X(x,x_0).
 \]
 \item Assume additionally that $\X$ is a compact subset of a Banach space $(\Y,\|\cdot\|)$.\footnote{So that $\dd_{\X}(x,x') = \|x-y\|$ for all $x,x'\in \X$} The \emph{circumradius} of $\X$ in $\Y$ is
 \[
 \rad(\X \subset \Y) := \inf_{y\in \Y} \sup_{x\in \X} \|x-y\|.
 \]
\end{enumerate}
\end{defn}

\section{Bounding Lifespans via Widths} \label{sec:width}

In this section, we recall and establish upper bounds for the filling radius which can be used to estimate the lifespans in both \v{C}ech and Vietoris--Rips settings. The basic idea behind many constructions is simple and can be illustrated as follows. Let $\X$ be a subset of a Banach space $\Y$, and let $k > \ell$ be positive integers. A degree-$k$ homology class of $\X$ can be ``killed'' by deforming $\X$ to an $\ell$-dimensional complex inside $\Y$ (a ``core''), and if every point moves by some controlled distance, then we obtain an estimate for the filling radius of the degree-$k$ homology. 

\subsection{Background on Widths} \label{sec:width-background}

Here we recall three classical approaches to measuring approximate dimension. Informally, the $k$-width of a space $\X$ measures the extent to which $\X$ fails to be $k$-dimensional.

The Urysohn width was historically the first one to be introduced. The following definition is equivalent to the one given by Urysohn around 1923 in the context of dimension theory; it was posthumously published by Alexandrov~\cite{alexandroff1926notes}.

\begin{defn} [Urysohn Width] \label{def:Urysohn} Let $\X$ be a compact metric space. For an integer $k \geq 0$, the \textit{Urysohn $k$-width} of $\X$ is defined as $$\UW_k(\X)=\inf_f \sup_p \diam(f^{-1}(p)),$$ where $f:\X\to \Delta^k$ is any continuous map to any finite $k$-dimensional simplicial complex.
\end{defn}

By definition,  widths are monotone in the dimension parameter: $$\UW_0(\X)\geq \UW_1(\X) \geq \UW_2(\X) \geq \dots,$$ and $\UW_0(\X) := \diam(\X)$ if $\X$ is connected. The $n$-width of an $n$-dimensional manifold is zero, and all preceding $k$-widths are positive for $0 \leq k < n$. 

The crucial connection between the Urysohn width and the filling radius is due to Gromov.

\begin{thm}[{\cite{gromov1983filling}}] For any closed $n$-dimensional Riemannian manifold $\M$, $$\rho(\M) \le \frac12 \UW_{n-1}(\M).$$ 
\end{thm}
The homology coefficients in the definition of $\rho(\M)$ do not matter here; everywhere below we assume them to lie in an arbitrary field, which is fixed and omitted from the notation. The importance of this result for bounding lifespans becomes immediate once one notices that the proof of this inequality can be easily generalized for any homological feature, and not just the fundamental class.

\begin{thm}
\label{thm:fillrad-urysohn-width} 
For any compact metric space $\X$, any integer $k \geq 1$, and any homology class $\omega \in \Hom_{k}(\X)$, 
\[
\rho(\omega; \X) \leq \frac12 \UW_{k-1}(\X).
\]
\end{thm}

\begin{proof}[Proof sketch, following~\cite{gromov2020no}]
Let $f:\X\to \Delta^{k-1}$ be a map for which $\delta = \sup_p \diam(f^{-1}(p))$ is just a tiny amount bigger than $\UW_{k-1}(\X)$. Consider the cylinder $\X \times [0,\delta/2]$ and glue its end $\X\times\{\delta/2\}$ to a copy of $\Delta^{k-1}$ along the map $f$; that means, pinch every fiber of $f$ inside $\X\times\{\delta/2\}$ to a point. The resulting space $\C_f$ can be endowed with a metric that restricts on $\X\times \{0\}$ to the original metric of $\X$, and makes the length of every interval $\{x\} \times [0,\delta/2]$ equal to $\delta/2$. Now, any $k$-cycle $S$ in $\X\times \{0\}$ representing $\omega$ becomes null-homologous in $\C_f$ (informally, one can just slide it towards the pinched end of the cylinder, where it degenerates to a $(k-1)$-dimensional set). Next, we embed $\C_f$ to $L^\infty(\C_f)$ in a distance-preserving way, and post-compose it with the $1$-Lipschitz restriction $L^\infty(\C_f) \to L^\infty(\X)$, corresponding to the inclusion $\X\times\{0\} \subset \C_f$. The $(k+1)$-chain that we built in $\C_f$ to fill $S$ pushes forward to $L^\infty(\X)$, and there it lies within distance $\delta/2$ of the Kuratowski image of $\X$.
\end{proof}

It is natural to look for a relative version of Theorem~\ref{thm:fillrad-urysohn-width}. We provide such an inequality below (Theorem~\ref{thm:fillrad-alexandrov-width}), additionally replacing its right-hand side by the quantity that is comparable to the Urysohn width but easier to compute; the corresponding metric invariant implicitly appeared in the work of  Alexandrov~\cite{alexandroff1933urysohnschen} on dimension theory.

\begin{defn} [Alexandrov Width] \label{defn:AW} Let $\X$ be a compact subset of a Banach space $(\Y, \|\cdot\|)$. For an integer $k \geq 0$, the \textit{Alexandrov $k$-width} of $\X$ (relative to $\Y$) is defined as 
\[
\AW_k(\X \subset \Y):=\inf_f \sup_{x \in \X} \|x-f(x)\|,
\]
where $f:\X\to \Y$ is any continuous map whose image is a finite simplicial complex of dimension at most $k$.
\end{defn}

The Alexandrov width enjoys similar properties with the Urysohn width: it is monotonically decreasing in $k$ until it reaches zero when $k$ becomes equal the dimension of $\X$. If $\X$ is connected, $\AW_0(\X \subset \Y)$ equals the circumradius of $\X$ in $\Y$, that is, $\AW_0(\X \subset \Y) = \rad(\X \subset \Y)$ (see \Cref{def:radius}).

It is easy to see directly from the definitions that $\UW_k(\X) \leq 2\AW_k(\X\subset\Y)$. Combining this with Theorem~\ref{thm:fillrad-urysohn-width}, we obtain the estimate $\rho(\omega; \X) \leq \AW_{k-1}(\X\subset\Y)$, which we will improve by replacing the absolute filling radius by the relative one.

\begin{thm}
\label{thm:fillrad-alexandrov-width} 
For any compact set $\X$ sitting in a Banach space $\Y$, any integer $k \geq 1$, and any homology class $\omega \in \h_{k}(\X)$, 
\[
\rho(\omega; \X\subset\Y) \leq \AW_{k-1}(\X\subset\Y).
\]
\end{thm}

This is the relative counterpart of Theorem~\ref{thm:fillrad-urysohn-width}. It will follow from a stronger estimate below (Theorem~\ref{thm:fillrad-treewidth}), but the intuition behind it is simple as described at the preamble of this section: One can kill higher homology of $\X$ by deforming it to a low-dimensional complex in $\Y$, and if every point moves by some controlled distance, then we obtain an estimate for the filling radius.

The virtues of \Cref{thm:fillrad-alexandrov-width} are twofold. First, the Alexandrov width seems to be easier to estimate than the Urysohn width.
Basically, the Urysohn width considers all maps to a $k$-dimensional space, whereas the Alexandrov width only considers a special class of those maps with the images lying in the same ambient space $\Y$; this is also the reason why $\UW_k(\X) \leq 2\AW_k(\X\subset\Y)$. Second, in \Cref{sec:app-AW-UW} we explain that $\AW_k(\X\subset\Y) \leq \UW_k(\X)$. Hence, together with \Cref{thm:fillrad-alexandrov-width}, this implies one can bound the relative filling radius with the Urysohn width (if $\X$ is compact), too, and we do not lose much when substituting the widths, since
$$\AW_k(\X\subset\Y) \leq \UW_k(\X) \leq 2\AW_k(\X\subset\Y).$$

When $\Y$ is hyperconvex (for example, $L^\infty(\X)$), one can claim more (see \Cref{rmk:aw=uw}):
$$\UW_k(\X) = 2\AW_k(\X\subset\Y).$$

While Urysohn's and Alexandrov's notions measure \emph{non-linear} width, a simpler concept of \emph{linear} width appeared in the work of Kolmogorov~\cite{kolmogoroff1936uber} in the context of approximation theory. It is commonly used in infinite-dimensional settings. 

\begin{defn} [Kolmogorov Width] \label{defn:KW} 
Let $\X$ be a compact subset of a Banach space $\Y$. For an integer $k \geq 0$, the \textit{Kolmogorov $k$-width} of $\X$ (relative to $\Y$) is defined as 
\[
\KW_k(\X \subset \Y) := \inf \{r ~\vert~ \X \subset \N_r(\p \subset \Y) \text{ for some affine $k$-plane } \p \subset \Y\}.
\]
\end{defn}

Again, these quantities monotonically decrease in $k$. It is immediate from definitions that $\AW_k(\X\subset\Y) \le \KW_k(\X\subset\Y)$. From Theorem~\ref{thm:fillrad-alexandrov-width} it follows then the relative filling radius of degree-$k$ homology can be bounded as $\rho(\cdot; \X\subset\Y) \leq \KW_{k-1}(\X\subset\Y)$. The following finer estimate is another corollary of Theorem~\ref{thm:fillrad-treewidth}.

\begin{thm}
\label{thm:fillrad-kolmogorov-width}
Let $\X$ be a compact subset of a Banach space $\Y$. Then for any homology class $\omega \in \Hom_{k}(\X)$, $k\ge 0$, 
\[
\rho(\omega; \X \subset \Y) \leq \KW_k(\X\subset\Y).
\]
\end{thm}

\begin{rmk} [Filling radius and $k$-widths] \label{rmk:ellipsoid-high} We give a toy example to illustrate the relation between the $k$-width and filling radius. 
 
 For $m\leq N$, consider the $(m-1)$-dimensional ellipsoid 
 $\E\subset \R^N$ (with the extrinsic metric) given by
 $$\E:=\left\{\mathbf{x}\in \R^N ~\middle\vert~ \sum_{i=1}^{m}\dfrac{x_i^2}{a_i^2}=1, \mbox{ and } x_j=0 \mbox{ for } j>m\right\},$$ where $a_1>a_2>\dots>a_{m}>0$. The widths can be roughly estimated as $\UW_k(\E) \sim \AW_k(\E \subset \R^N) \sim a_{k+1}$ for $0\leq k<m$. 
 On the other hand, the relative filling radius of $\E$ in $\R^N$ equals the length of the shortest axis: $\rho([\E];\E\subset \R^N) = a_{m}$. Here, the dimensions $m,N$ and the sequence $\{a_i\}$ are arbitrary. \Cref{thm:fillrad-alexandrov-width} tells us that $\rho([\E];\E\subset \R^N) \le \AW_{m-2}(\E \subset \R^N)$, but we see that the discrepancy between $\rho([\E];\E\subset \R^N)= a_{m}$ and $\AW_{m-2}(\E \subset \R^N) \sim a_{m-1}$ can be arbitrarily large. In other words, the Alexandrov (and Urysohn) width can highly overestimate the filling radius. However, the estimate of \Cref{thm:fillrad-kolmogorov-width} in this example is sharp:
$\rho([\E];\E\subset \R^N)= \KW_{m-1}(\E \subset \R^N) = a_m$ \end{rmk}

For more information on widths, we refer the reader to Balitskiy~\cite{balitskiy2021bounds}.

\subsection{Bounds for \v{C}ech Lifespans via Treewidth} \label{sec:treewidth} 

Recall (from formula~\eqref{eq:lifespan-filling-radius} in \Cref{sec:FR_background}) that $\mathrm{Spec}_k(\C_\bullet(\X\subset\Y))$-lifespans can be written in terms of relative filling radii as follows: 
\[
d_\omega - b_\omega =  \lim_{r \to b_\omega + 0} \rho(\omega_{r}; \N_{r}(\X\subset\Y) \subset \Y),
\]
where $\omega_r$ maps to $\omega$ by the homology map induced by the inclusion (see \Cref{sec:FR_background}). To estimate lifespans, we need upper bounds for filling radii, such as in \Cref{thm:fillrad-alexandrov-width} and \Cref{thm:fillrad-kolmogorov-width}. Both estimates can be simultaneously strengthened using the following new width invariant. Recall that throughout the paper, we use homology with coefficients in a fixed field, which is omitted from the notation.

\begin{defn}[Treewidth] \label{defn:treewidth}
Let $\X$ be a compact subset of a Banach space $(\Y, \|\cdot\|)$. For an integer $k \geq 0$, let us define the \emph{$k\textsuperscript{th}$ treewidth} of $\X$ as 
\[
\TW_k(\X\subset\Y):=\inf_f \sup_{x \in \X} \|x - f(x)\|,
\]
where $f:\X\to \Y$ is any continuous map whose image $f(\X)$ is a finite simplicial complex of dimension at most $k$ with trivial $k\textsuperscript{th}$ reduced homology. 
\end{defn}

The treewidth satisfies evident properties: 
\[
\rad(\X\subset \Y) = \TW_0(\X\subset \Y) \ge \TW_1(\X\subset \Y) \ge \TW_2(\X\subset \Y) \ge \ldots.
\]
An equivalent definition of the $k\textsuperscript{th}$ treewidth (perhaps, better motivating our choice of the word \emph{treewidth}) is obtained if one only considers maps $f$ such that the image $f(\X)$ is a subcomplex of a finite $k$-dimensional \emph{contractible} complex inside $\Y$. Indeed, to any $k$-complex with trivial $\Hom_k(\cdot)$ one can glue several cells of dimension $\leq k$ to kill all its homology in lower degrees as well as its fundamental group. In the opposite direction, any subset of a $k$-dimensional contractible complex has trivial $\Hom_k(\cdot)$. 

The motivation behind the definition of treewidth is to give a common generalization of \Cref{thm:fillrad-alexandrov-width} and \Cref{thm:fillrad-kolmogorov-width}. Both follow from the following result combined with the evident bounds $\TW_{k}(\X\subset\Y) \leq \AW_{k-1}(\X\subset\Y)$ and $\TW_{k}(\X\subset\Y) \leq \KW_{k}(\X\subset\Y)$.

\begin{thm}
\label{thm:fillrad-treewidth}
For any compact set $\X$ sitting in a Banach space $\Y$, any integer $k \geq 0$, and any homology class $\omega \in \Hom_{k}(\X)$, 
\[
\rho(\omega; \X\subset \Y) \leq \TW_k(\X\subset\Y).
\]
\end{thm}

\begin{proof}
Fix any number $\delta > \TW_k(\X\subset\Y)$, and pick a witness map $f : \X \to \Y$, whose image $f(\X) = \T$ has dimension at most $k$ and trivial $\Hom_k(\T)$, such that $\|x-f(x)\| \leq \delta$ for all $x \in \X$. Pick a cycle $S \subset \X$ representing $\omega$ and continuously deform it via the linear homotopy $h: S \times [0,1] \to \Y$ given by
\[
h(x,t) = (1-t)x + tf(x).
\]
Every point $x \in S$ moves by distance at most $\delta$, so the continuous deformation always stays in $\N_{\delta}(S\subset\Y)$. At time $t=1$, the deformed cycle $h(S,1)$ lies in $\T$ and bounds a $(k+1)$ chain within $h(S,1)$, because $\T$ is of dimension $k$ with trivial $\Hom_k(\T)$. Therefore, $S$ is null-homologous in $\N_{\delta}(S\subset\Y)$ (basically, $S$ bounds the $(k+1)$-dimensional trace of the homotopy $h$). 
\end{proof}

A simple example in which \Cref{thm:fillrad-treewidth} gives a stronger bound than \Cref{thm:fillrad-alexandrov-width} and \Cref{thm:fillrad-kolmogorov-width} is depicted in \Cref{fig:tripod}; the fundamental class of this circle has a small filling radius, which can be efficiently estimated by the $1$-treewidth (approximating the shape by a tree, red in the figure), while approximations by straight lines or points give significantly worse bounds.

\begin{figure}
    \centering
    \includegraphics[width=0.3\textwidth]{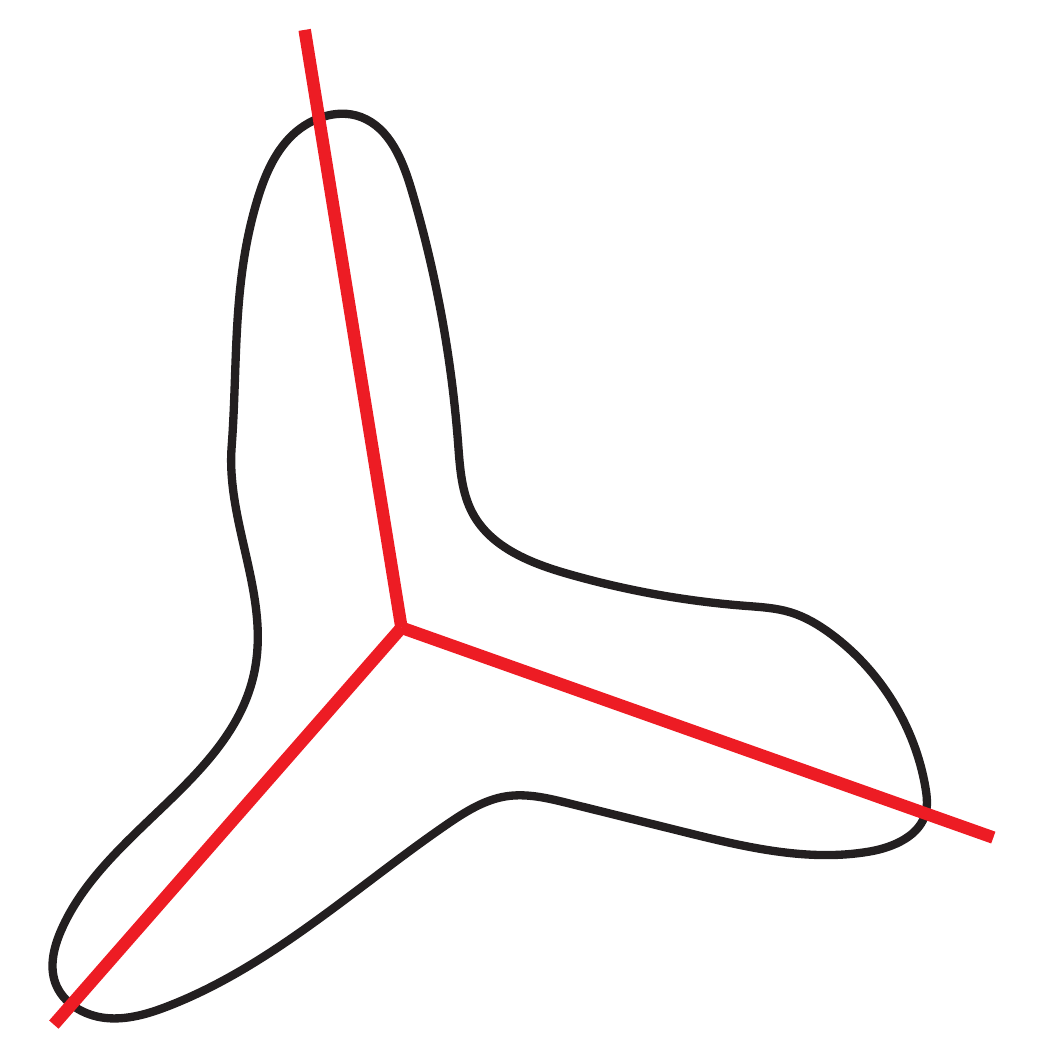}
    \caption{A shape with small $1$-treewidth but considerable Alexandrov $0$-width and Kolmogorov $1$-width.} 
    \label{fig:tripod}
\end{figure}

We proceed with a strengthening applicable to non-compact sets of the form $\N_r(\X \subset \Y)$. Notice that the width in the right-hand side of the following estimate is computed on a compact set, since in Banach spaces the closure of the convex hull of a compact set is compact; see e.g. Lax~\cite[Section~13, Exercise~9]{lax2002functional}. 

\begin{thm}
\label{thm:fillrad-treewidth-conv}
For any compact set $\X$ sitting in a Banach space $\Y$, any integer $k \geq 0$, any positive number $r$, and any homology class $\omega \in \Hom_{k}(\N_r(\X \subset \Y))$, 
\[
\rho(\omega; \N_r(\X \subset \Y) \subset \Y) \leq \TW_k(\cN_r(\X \subset \cconv(\X)) \subset\Y).
\]

\end{thm}

\begin{proof}
Let our homology class $\omega \in \Hom_{k}(\N_r(\X\subset\Y))$ be represented by a singular cycle $S_0$ supported in $\N_r(\X\subset\Y)$. Since $S_0$ is a compact subset of $\Y$, it can be covered by finitely many open balls $B_r(x), x \in \X' \subset \X, |\X'| < \infty$ (recall that $B_r(x)$ denotes the open ball of radius $r$ centered at $x$). Consider the union $U = \bigcup\limits_{x \in \X'} B_r(x)$ and choose any partition of unity subordinate to the cover of $U$ by $B_r(x), x \in \X'$. This partition of unity gives rise to a map from $U$ to the nerve $L$ of that cover. By the nerve theorem, this map $\phi: U \to L$ is a homotopy equivalence, with an evident homotopy inverse $\psi: L \to U$, which sends every vertex of $L$ to the center of the corresponding ball in the cover $U = \bigcup\limits_{x \in \X'} B_r(x)$, and extends affinely on the rest of $L$. Therefore, $\omega\vert_U = \psi_* \circ \phi_* (\omega\vert_U) \in \Hom_{k}(U)$, and the class $\omega$ can be represented by a cycle $S$ supported in $\psi(L) \subset \cconv(\X)$. A similar argument shows that if $S$ bounds in $\N_R(\X \subset \Y)$, for some $R > r$, then the filling chain can also be taken with the support in $\cconv(\X)$. Therefore,
\[
\rho(\omega; \N_r(\X \subset \Y) \subset \Y) \le \rho(\omega; \N_r(\X \subset \cconv(\X)) \subset \Y).
\]
Now we are in position to apply \Cref{thm:fillrad-treewidth} to the compact set $\cN_r(\X \subset \cconv(\X))$:
\begin{align*}
\rho(\omega; \N_r(\X \subset \cconv(\X)) \subset \Y) &= \rho(\omega; \cN_r(\X \subset \cconv(\X)) \subset \Y) \\
&\le \TW_k(\cN_r(\X \subset \cconv(\X)) \subset\Y).
\end{align*}
\end{proof}

\begin{cor}[\v{C}ech Lifespans via Treewidth]\label{cor:cech-TW} Let $\X$ be compact subset of a Banach space $\Y$ and let $\omega\in\mathrm{Spec}_k(\C_\bullet(\X\subset \Y))$, $k\ge 0$. Then, 
\[
d_\omega-b_\omega\leq \TW_k(\cN_{b_\omega}(\X \subset \cconv(\X)) \subset\Y).
\]
In particular, $$d_\omega-b_\omega\leq \TW_k(\cconv(\X) \subset\Y).$$
\end{cor}

\begin{rmk}
Recall the monotonicity of the treewidth: $\TW_k(\cdot) \le \TW_0(\cdot) = \rad(\cdot)$. Therefore, as a trivial consequence of~\Cref{cor:cech-TW}, we can upper-bound lifespans in all dimensions by $\TW_0(\cconv(\X) \subset\Y) = \rad(\X \subset \Y)$.    
\end{rmk}

\begin{proof}[Proof of~\Cref{cor:cech-TW}] 
We use the aforementioned formula
\[
d_\omega - b_\omega =  \lim_{r \to b_\omega + 0} \rho(\omega_{r}; \N_{r}(\X\subset\Y) \subset \Y),
\]
and bound each filling radius by the treewidth of $\N_{r}(\X\subset\Y)$:

\[
\rho(\omega_{r}; \N_{r}(\X\subset\Y) \subset \Y) \le \TW_k(\cN_r(\X \subset \cconv(\X)) \subset\Y).
\]
To conclude, we need a continuity property:
\[
\lim_{r \to b_\omega + 0} \TW_k(\cN_r(\X \subset \cconv(\X)) \subset\Y) = \TW_k(\cN_{b_\omega}(\X \subset \cconv(\X)) \subset\Y).
\]
Its proof is explained in the appendix; see \Cref{lem:treewidth-continuity}.
\end{proof}

Combining this with the trivial inequality $\TW_{k}(\cdot) \leq \AW_{k-1}(\cdot)$, we obtain the following bound.

\coraw

\begin{figure}
 \begin{center}
 \includegraphics[width=0.8\textwidth]{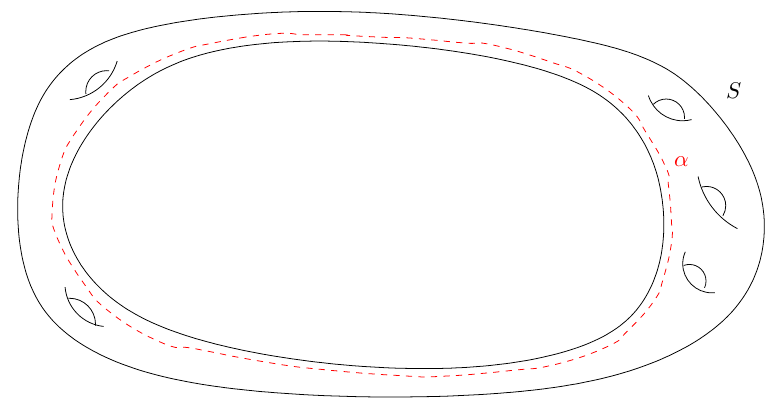}
 \caption{For the surface $S$, while $\AW_1(S \subset \mathbb{R}^3)$ is small, the lifespan of red curve $\alpha\in \Hom_1(S)$ is large. \label{fig:lowerdim}}
 \end{center}
\end{figure} 

We would like to point out that while $k$-widths are effective for bounding lifespans of homology classes of degree $> k$, they do not say much about the lifespans of homology classes in lower degrees. In \Cref{fig:lowerdim}, we give a simple $2$-dimensional example $S$. Here, $\AW_1(S \subset \mathbb{R}^3)$ is small, and hence, the homology class in $\Hom_2(S)$ has short lifespans by \Cref{cor:cech-AW}. However, one can easily see that homology classes $\alpha\in \Hom_1(S)$ can still have very large lifespans. One can generalize this example to any dimension and codimension. In \Cref{sec:extinction}  we take a different perspective, and give global bounds for lifespans and death times in all degrees simultaneously.

\subsection{Bounds for Vietoris--Rips Lifespans via Urysohn Width} \label{sec:VR_Urysohn}

Recall the relationship between $\mathrm{Spec}_k(\VR_\bullet(\X))$-lifespans and filling radii implied by formula~\eqref{eq:lifespan-filling-radius} of \Cref{sec:FR_background}:
\[
d_\omega - b_\omega = 2 \lim_{r \to b_\omega + 0} \rho\big(\omega_{r}; \N_{r}(\kappa(\X)\subset L^\infty(\X)) \subset L^\infty(\X)\big),
\]
where $\omega_r$ maps to $\omega$ by the homology level map induced by the inclusion, and $\kappa: \X \to L^\infty(\X)$ is the Kuratowski embedding. Recall also (\Cref{rmk:vr-filling-redundancy}) that the right-hand side can be rewritten a bit shorter:
\[
\rho\big(\omega_{r}; \N_{r}(\kappa(\X)\subset L^\infty(\X)) \subset L^\infty(\X)\big) = \rho\big(\omega_{r}; \N_{r}(\kappa(\X)\subset L^\infty(\X))\big).
\]

This follows from the following lemma.

\begin{lem}\label{lem:fillrad_redundancy}
Let $\mathcal{Z}$ be a subset of $L^\infty(\X)$ containing the Kuratowski image $\kappa(\X)$. Then, for any homology class $\omega \in \Hom_k(\mathcal{Z})$, the following is true:
\[
\rho(\omega; \mathcal{Z} \subset L^\infty(\X)) = \rho(\omega; \mathcal{Z}).
\]
\end{lem}

\begin{proof}
The inequality ``$\ge$'' follows from the universal property of the absolute filling radius (see the comment after \Cref{defn:homology_filling_radius}). To show the inequality ``$\le$'', we first recall that
\[
\rho(\omega; \mathcal{Z}) = \rho(\kappa'_*(\omega); \kappa'(\mathcal{Z}) \subset L^{\infty}(\mathcal{Z})),
\]
where $\kappa' : \mathcal{Z} \to L^{\infty}(\mathcal{Z})$ is the Kuratowski map. Second,
we make use of the injectivity (\Cref{def:Injme}) of $L^{\infty}(\X)$ in the following way. The inclusion $\mathcal{Z} \subset L^\infty(\X)$, via the Kuratowski identification $\mathcal{Z} \overset{\sim}\to \kappa'(\mathcal{Z})$, gives rise to a distance preserving (hence, $1$-Lipschitz) map $\kappa'(\mathcal{Z}) \to L^\infty(\X)$. This map can be extended to a $1$-Lipschitz map $p : L^{\infty}(\mathcal{Z}) \to L^\infty(\X)$, by invoking the injectivity of $L^{\infty}(\X)$. This map pushes forward the cycle $\kappa'_*(\omega)$ to $\omega$. Moreover, it pushes forward any chain filling $\kappa'_*(\omega)$ inside $\N_r(\kappa'(\mathcal{Z})) \subset L^{\infty}(\mathcal{Z})$, to a chain filling $\omega$ inside $\N_r(\mathcal{Z} \subset L^\infty(\X))$. 
\end{proof}

To estimate VR-lifespans, we need to know how to upper-bound quantities like $\rho\big(\cdot; \N_{r}(\kappa(\X)\subset L^\infty(\X))\big)$, for which Theorem~\ref{thm:fillrad-urysohn-width} is not applicable: the set $\N_{r}(\kappa(\X)\subset L^\infty(\X))$ is not compact (even after taking the closure). One way to deal with non-compactness is to intersect $\N_r(\kappa(\X)\subset L^\infty(\X))$ with the closure of the convex hull of $\X$, like we did in \Cref{thm:fillrad-treewidth-conv}.  That theorem, together with the inequality $\TW_k(\cdot) \le \AW_{k-1}(\cdot)$, will give us a bound on $\rho\big(\cdot; \N_{r}(\kappa(\X)\subset L^\infty(\X))\big)$ in terms of the Alexandrov width of $\cN_r\big(\kappa(\X)\subset \cconv(\kappa(\X))\big)$. It turns out the latter can be interpreted in terms of the Urysohn width.

\begin{rmk}[$L^\infty$-interpretation of Urysohn width]\label{rmk:fillrad-width} In the inequality $\rho(\omega; \X) \leq \frac12 \UW_{k-1}(\X)$, the left-hand side is defined via the Kuratowski embedding in $L^\infty(\X)$, and it is instructive to interpret the right-hand side as well, in a way compatible with the picture of $\kappa(\X)$ sitting in $L^\infty(\X)$. This interpretation goes back to Tikhomirov~\cite{tikhomirov1976some}. Some references in English are~\cite[Appendix~1, Proposition~(D)]{gromov1983filling} by Gromov and~\cite[Theorem~2.1.9 and Definition~2.3.1]{balitskiy2021bounds} by Balitskiy. Here we provide a short summary. The width $\UW_{k-1}(\X)$ can be equivalently defined as the infimal number $\delta>0$ such that there is map $f: \X \to L^\infty(\X)$ whose image is at most $(k-1)$-dimensional complex in $L^\infty(\X)$, and such that $\|\kappa(x) - f(x)\|_\infty \le \delta/2$ for all $x \in \X$. In a nutshell, Theorem~\ref{thm:fillrad-urysohn-width} says that one can try killing higher homology of $\X$ by deforming $\kappa(\X)$ to a low-dimensional complex inside $L^\infty(\X)$, and if every point moves by some controlled distance, then we get an estimate for the filling radius. This interpretation can be shortly reformulated as follows:
\[
\UW_{k-1}(\X) = 2\AW_{k-1}(\kappa(\X) \subset L^\infty(\X)).
\]
\end{rmk}

Given that interpretation, the following estimate is immediate:
\[
\rho\big(\omega; \N_{r}(\kappa(\X)\subset L^\infty(\X))\big) \leq \frac12 \UW_{k-1}\big(\cN_r(\kappa(\X)\subset \cconv(\kappa(\X)))\big).
\]
We will go further and give an even better estimate for $\rho\big(\cdot; \N_{r}(\kappa(\X)\subset L^\infty(\X))\big)$, by replacing the convex hull with the tight span (\Cref{defn:tightspan}). The following result subsumes Theorem~\ref{thm:fillrad-urysohn-width} when $r=0$. The specific realization of $\ts(\X)$ in $L^\infty(\X)$ that was described in \Cref{sec:ts} has an isometric copy of $\X$ inside, and the corresponding embedding of $\X$ in $L^\infty(\X)$ agrees with the Kuratowski embedding. 

\medskip
To simplify notation a bit, in the rest of this section, we omit $\kappa$ and simply write $\X \subset \ts(\X) \subset L^\infty(\X)$. Notice that the set $\overline{\N}_r(\X\subset \ts(\X))$ is compact, as a closed subset of the compact metric space $\ts(\X)$.

\begin{thm} \label{thm:fillrad-urysohn-width-injective} 
For any compact metric space $\X$, any integer $k \geq 1$, any positive number $r$, and any homology class $\omega \in \h_{k}(\N_r(\X\subset L^\infty(\X)))$, 
\[
   \rho\big(\omega; \N_{r}(\X \subset L^\infty(\X))\big) \le \frac12 \UW_{k-1}\big(\overline{\N}_r(\X\subset \ts(\X))\big).
\]
\end{thm}

\begin{proof}
We have a chain of distance-preserving embeddings $\X \subset \ts(\X) \subset L^\infty(\X)$ as well as $\X \subset \N_r(\X\subset \ts(\X)) \subset \N_r(\X \subset L^\infty(\X))$. All distances below are measured using the $L^\infty$-norm. We need to bound $\rho\big(\omega; \N_r(\X \subset L^\infty(\X)) \subset L^\infty(\X)\big)$. 

Recall the following property of the tight span (\Cref{rmk:tightretraction}): there exists a $1$-Lipschitz retraction $\pi : L^\infty(\X) \to \ts(\X)$. Our first observation is that $\pi$ maps $\N_r(\X \subset L^\infty(\X))$ to $\N_r(\X \subset \ts(\X))$. Indeed, take any point $y \in \N_r(\X \subset L^\infty(\X))$. Since $\|y-x\|<r$ for some $x \in \X$, and $\pi$ is $1$-Lipschitz retraction, it follows that $\|\pi(y) - x\| = \|\pi(y) - \pi(x)\| \le \|y-x\|<r$. Therefore, $\pi\big(\N_r(\X \subset L^\infty(\X))\big) = \N_r(\X \subset \ts(\X))$. Denote the reverse inclusion $$\iota : \N_r(\X \subset \ts(\X)) \hookrightarrow \N_r(\X \subset L^\infty(\X)).$$ Since both $y$ and $\pi(y)$ lie in the open radius $r$ ball centered at $x$, the whole straight line segment between them lies in that ball and in $\N_{r}(\X \subset L^\infty(\X))$. Therefore, a cycle representing $\omega$ can be continuously deformed to a cycle representing $\iota_* \circ \pi_* (\omega)$ by letting each $y$ in the support of $\omega$ slide along the straight line segment towards $\pi(y)$. This homotopy takes place entirely in $\N_{r}(\X \subset L^\infty(\X))$, and therefore, $\omega = \iota_* \circ \pi_* (\omega)$.

It should be obvious now that
\[
\rho\big(\omega; \N_r(\X \subset L^\infty(\X)) \subset L^\infty(\X)\big) \le \rho\big(\pi_*( \omega); \N_r(\X \subset \ts(\X)) \subset L^\infty(\X)\big).
\]
Indeed, any cycle $S$ representing $\pi_* (\omega) \in \Hom_{k}\big(\N_r(\X \subset \ts(\X))\big)$ lies entirely in $\N_r(\X \subset L^\infty(\X))$ and there represents the homology class $\iota_* \circ \pi_* (\omega) = \omega \in \Hom_{k}\big(\N_r(\X\subset L^\infty(\X))\big)$. If $S$ bounds in the neighborhood of $\N_r(\X \subset \ts(\X))$ of certain radius, then it bounds even in a smaller neighborhood of $\N_{r}(\X \subset L^\infty(\X))$.

Using the observation at the beginning of this subsection, we deduce that
 \[
  \rho\big(\pi_* (\omega); \N_r(\X \subset \ts(\X)) \subset L^\infty(\X)\big) = \rho\big(\pi_* (\omega); \N_r(\X \subset \ts(\X))\big).
 \]

The next step follows trivially from the definition of the filling radius:
\[
\rho\big(\pi_*( \omega); \N_r(\X \subset \ts(\X))\big) = \rho\big(\pi_* (\omega); \overline{\N}_r(\X \subset \ts(\X))\big).
\]

The final step is to apply Theorem~\ref{thm:fillrad-urysohn-width} to the compact set $\overline{\N}_r(\X \subset \ts(\X))$, to conclude that
\[
\rho\big(\pi_*( \omega); \overline{\N}_r(\X \subset \ts(\X))\big) \le \frac12 \UW_{k-1}\big(\overline{\N}_r(\X \subset \ts(\X))\big).
\]
Assembling all these inequalities together, we obtain the result. 
\end{proof}

\corvruw

\begin{rmk}
It follows that lifespans in dimensions $1$ and higher can be upper-bounded by
$\UW_0(\ts(\X)) = \diam(\X)$. It is easy to see that lifespans in dimension $0$ are also bounded above by $\diam(\X)$.
\end{rmk}

\begin{proof}[Proof of~\Cref{cor:VR-UW}] Recall that 
\[
d_\omega - b_\omega = 2 \lim_{r \to b_\omega + 0} \rho\big(\omega_{r}; \N_{r}(\X\subset L^\infty(\X)) \subset L^\infty(\X)\big),
\]
where $\omega_r$ maps to $\omega$ by the homology map induced by the inclusion. Upper-bounds for these filling radii are given by Theorem~\ref{thm:fillrad-urysohn-width-injective}. To conclude, it remains to use the continuity property of the Urysohn width~\cite[Theorem~2.4.1]{balitskiy2021bounds}:
\[
\lim_{r \to b_\omega + 0}  \UW_{k-1}\big(\overline{\N}_r(\X\subset \ts(\X))\big) = \UW_{k-1}\big(\overline{\N}_{b_\omega}(\X\subset \ts(\X))\big).
\]
\end{proof}

\subsection{Robust Bounds for \v{C}ech Lifespans via Treewidth} \label{sec:robust}

The estimates provided in the preceding sections have a drawback: there is no a priori method to directly relate the width of $\X \subset \Y$ to the width of its $r$-neighborhood $\N_r(\X \subset \Y)$. This poses a challenge, as the bounds in Corollaries~\ref{cor:cech-TW}, \ref{cor:cech-AW}, and \ref{cor:VR-UW} rely on the width of a specific neighborhood of $\X$. However, by further refining the class of potential ``cores", we can ensure the derivation of more effective bounds. For example, this is possible if the core is an affine $k$-dimensional subspace (see \Cref{sec:PCA}).

Here, we introduce a modified version of treewidth that enables bounding lifespans directly in terms of $\X$ itself, rather than its neighborhood. To this end, we introduce an auxiliary definition to help constrain the geometric complexity of potential ``cores".

\begin{defn}
\label{defn:robust-core}
Let $(\Y,\|\cdot\|)$ be a Banach space, and let $\T \subset \Y$ be a simplicial complex. Suppose there is a retraction map $f: \Y \to \T$. We say that $f$ is \emph{$C$-robust} if for any finite set $\X \subset \Y$, and any $x \in \conv(\X)$,
\[
\|f(x)-x\| \leq \max_{x' \in \X} \|f(x') - x'\| + C \cdot \rad(\X\subset\Y).
\]
\end{defn}

\begin{rmk}
Recall that, given $\T\subset \Y$, a retraction $f: \Y \to \T$ is a continuous map fixing every point of $\T$. Since $\Y$ is contractible, the existence of a retraction implies that $\T$ is contractible.
\end{rmk}

\noindent {\bf Examples.} 
\begin{enumerate}
 \item If $\T$ is a closed affine subspace of $\Y$, and $f$ is any linear projection on $\T$, then $f$ is $0$-robust. 
 \item If $\T$ is nonempty closed convex subset of $\Y$, and $f$ is a continuous nearest-point projection (metric projection) on $\T$, then $f$ is $0$-robust.\footnote{It is known that a nearest-point projection exists, if $\Y$ is reflexive, but it does not have to be continuous in general; see for example the discussion after \cite[Corollary~5.1.19]{megginson2012introduction}.}
 \item If the assignment $x \mapsto \|f(x)-x\|$ is $C$-Lipschitz, then $f$ is $C$-robust.
 \item If $\Y = \mathbb{R}^2$ is the Euclidean plane, and $\T$ is the union of $m\ge 3$ equiangular rays (or intervals) with a common endpoint, then there is a natural retraction to $\T$ which is $\frac{1}{\tan (\pi/m)}$-robust.
\end{enumerate}

\begin{defn}[Robust Treewidth]
\label{defn:robust-treewidth}
Let $\X$ be a subset of a Banach space $(\Y, \|\cdot\|)$. For an integer $k \geq 0$, and a real number $C \ge 0$, we say that the \emph{$C$-robust $k$-dimensional treewidth} of $\X$ is at most $\delta$, if there is a $C$-robust retraction $f: \Y \to \T$ to a finite simplicial complex $\T \subset \Y$ of dimension at most $k$, displacing every point of $\X$ by distance at most $\delta$:
\[
\|f(x)-x\| \leq \delta \quad \forall x \in \X.
\]
The infimum of the numbers $\delta$ satisfying this condition will be denoted by $\TW^C_k(\X\subset\Y)$.
\end{defn}

It is immediate from the definition that $\TW^{C_1}_k(\X \subset \Y) \ge \TW^{C_2}_k(\X \subset \Y)$ for any $C_1 \le C_2$, and that $\TW^C_k(\X \subset \Y) \ge \TW_k(\X \subset \Y)$ for any $C$.

\begin{thm}
\label{thm:fillrad-robust-treewidth}
For any compact set $\X$ sitting in a Banach space $\Y$, any integer $k \geq 0$, any real numbers $C \ge 0$ and $r>0$, and any homology class $\omega \in \Hom_{k}(\N_r(\X\subset \Y))$,
\[
\rho(\omega; \N_r(\X\subset \Y) \subset \Y) \leq \TW^C_k(\X\subset\Y) + Cr.
\]
\end{thm}

\begin{proof}
Suppose we are given a $C$-robust retraction $f : \Y \to \T$ to a $k$-dimensional simplicial complex $\T\subset \Y$, such that $\|x-f(x)\| \leq \delta$ for all $x \in \X$.

Let our homology class $\omega \in \Hom_{k}(\N_r(\X\subset\Y))$ be represented by a singular cycle $S_0$ supported in $\N_r(\X\subset\Y)$. Since $S_0$ is compact as a subset of $\Y$, it can be covered by finitely many open balls $B_r(x), x \in \X' \subset \X, |\X'| < \infty$ (here $B_r(x)$ denotes the open ball of radius $r$ centered at $x$). Consider the union $U = \bigcup\limits_{x \in \X'} B_r(x)$ and choose any partition of unity subordinate to the cover of $U$ by $B_r(x), x \in \X'$. This partition of unity induces  a map from $U$ to the nerve $L$ of that cover. By the nerve theorem this map $\phi: U \to L$ is a homotopy equivalence, with an evident homotopy inverse $\psi: L \to U$, which sends every vertex of $L$ to the center of the corresponding ball in the cover $U = \bigcup\limits_{x \in \X'} B_r(x)$, and extends affinely on the rest of $L$. Therefore, $\omega\vert_U = \psi_* \circ \phi_* (\omega\vert_U) \in \Hom_{k}(U)$, and the class $\omega$ can be represented by a cycle $S$ supported in $\psi(L)$; in other words, $S$ is obtained by gluing flat $k$-simplices with vertices in $\X$, and with the circumradius of each simplex at most $r$.

Now we continuously deform $S$ using the linear homotopy $h: S \times [0,1] \to \Y$ given by
\[
h(x,t) = (1-t)x + tf(x).
\]
By the robustness property of $f$, every point $x \in S$ is moved by a distance at most $\delta + Cr$, so the continuous deformation always stays in $\N_{\delta+Cr}(S\subset\Y)$. At time $t=1$, the deformed cycle $h(S,1)$ lies in $\T$ and bounds a $(k+1)$ chain within $h(S,1)$, because $\T$ is of dimension $k$ with trivial $\Hom_k(\T)$. Therefore, $S$ is nullhomologous in $\N_{\delta+Cr}(S\subset\Y)$ (basically, $S$ bounds the $(k+1)$-dimensional trace of the homotopy $h$).
\end{proof}

\cortw

\begin{proof} 
We again employ formula~\eqref{eq:lifespan-filling-radius}, now bounding filling radii by \Cref{thm:fillrad-robust-treewidth}:
\begin{align*}
d_\omega - b_\omega &=  \lim_{r \to b_\omega + 0} \rho(\omega_{r}; \N_{r}(\X\subset\Y) \subset \Y) \\
&\le \lim_{r \to b_\omega + 0}  \TW^C_k(\X\subset \Y) + Cr \\
&= \TW^C_k(\X\subset \Y) + C b_\omega.
\end{align*}
For $b_\omega \geq 1$, if we divide the entire inequality by $b_\omega$, we obtain the result.
\end{proof}

\begin{rmk} [Comments about \Cref{cor:cech-CTW}] We make the following remarks.
\begin{itemize}
    \item Notice that on the right-hand side, there is no reference to  the birth time $b_\omega$. Therefore, $C$-robust treewidth yields strong control on \emph{late-born} homology classes (i.e. those satisfying $b_\omega\geq 1$). 
    \item If the the assumption  $b_\omega\geq 1$ is changed to  $b_\omega \geq \tfrac{1}{\alpha}$ for some $\alpha>0$ then the conclusion becomes $\tfrac{d_\omega}{b_\omega}\leq 1+ C + \alpha \TW^C_k(\X\subset \Y)$.
\item The corollary  implies a robust estimate on the \emph{multiplicative persistence} of  intervals in the persistence barcode of the \v{C}ech filtration $\C_\bullet(\X\subset \Y)$. See  Bobrowski, Kahle and Skraba~\cite{bobrowski2017maximally} and Adams and Coskunuzer~\cite{adams2022geometric} for studies of multiplicative persistence.
\item Additionally,  the proof above  gives the following estimates for $(b_\omega,d_\omega)$:
\begin{itemize}
 \item If $b_\omega=0$, then $d_\omega\leq \TW^C_k(\X\subset \Y)$.
 \item If $b_\omega\geq \TW^C_k(\X\subset \Y)$, then $\dfrac{d_\omega}{b_\omega}\leq C+2$.
\end{itemize}
\end{itemize}
\end{rmk}

\subsection{Robust Bounds for \v{C}ech Lifespans via Kolmogorov Width and PCA$_\infty$} \label{sec:PCA}

For a given \emph{finite} set $\X$ in $\R^N$,  standard Principal Component Analysis or PCA (see \cite[\S14.5]{hastie2009elements}) approximates $\X$ by affine $k$-dimensional subspaces through minimizing over all such subspaces $\p$ the sum of squared 
distances of points in $\X$ to  $\p$. The infimum obtained through this procedure is called the $(k+1)\textsuperscript{th}$-\emph{variance} of $\X$.

Since its introduction by Pearson in 1901~\cite{pearson1901liii}, PCA has become one of the most widely used data analysis techniques, owing to its interpretability and the availability of highly efficient computational algorithms. In this work, we investigate a specific $\ell^\infty$-variant of classical PCA, designed for compact subsets of Banach spaces, and examine its connections to persistent homology by analyzing how it constrains the lifespans of homology classes.

\begin{defn} [PCA$_\infty$] \label{defn:pca-infty} Let $\X$ be a compact subset in a Banach space $(\Y,\|\cdot\|)$. For $0\leq k< \dim(\Y)$, let $\A_k$ be the space of affine $k$-dimensional subspaces in $\Y$. Then, we define the \textit{$(k+1)\textsuperscript{th}$ $\ell^\infty$-variance} of $\X$ as
\[
\nu_{k+1}(\X\subset\Y):=\inf_{\p\in \A_{k}}\sup_{x\in \X} \dd_\Y(x,\p)
\]
where $\dd_\Y(x,\p)$ is the distance from $x$ to the $k$-subspace $\p$, i.e., $\dd_\Y(x,\p):=\inf_{p\in\p}\|x-p\|$. 
\end{defn}

\begin{rmk}\label{rmk:kolmogorov}
The term $\ell^\infty$-variance is just a different name for Kolmogorov widths (cf. \Cref{defn:KW}):
\[
\nu_{k+1}(\X\subset \Y) = \KW_k(\X\subset \Y).
\]
\end{rmk}

\begin{thm}
\label{thm:fillrad-pca} 
Let $\X$ be a subset of a Banach space $\Y$. Then, for any $r>0$ and any homology class $\omega \in \Hom_{k}(\N_r(\X\subset \Y))$, $k\ge 0$, 
\[
\rho(\omega; \N_r(\X \subset \Y) \subset \Y) \leq \nu_{k+1}(\X \subset \Y) =  \KW_k(\X\subset \Y).
\]
\end{thm}

\begin{proof}
Apply Theorem~\ref{thm:fillrad-robust-treewidth} with $\T$ being a $k$-dimensional affine subspace, and $f$ being a $0$-robust nearest-point projection.
\end{proof}

 \v{C}ech Lifespans of homology classes in degree $\geq k$ cannot exceed the $(k+1)\textsuperscript{th}$ variance, $\nu_{k+1}(\X)$.

\corpca
\begin{proof} Combine formula~\eqref{eq:lifespan-filling-radius} with \Cref{thm:fillrad-pca}. 
\end{proof}

\begin{rmk} 
Notice that the corollary above gives strong estimates for point clouds in $\R^N$. For example, if the variance $\nu_{k_0+1}(\X)$ is small, then, in applications where only long-lived homological features are relevant,  the corollary  implies that one can ignore  homology classes  in $\mathrm{Spec}_k(\C_\bullet(\X\subset \Y))$ for $k\geq k_0$. We should note that, by its very definition,  PCA$_\infty$ is more sensitive than the original PCA against outliers. We discuss this difference and pose a statistical question about the relationship between the original PCA and persistent homology in \Cref{sec:remarks}. 
\end{rmk}

\begin{rmk} \label{rmk:general_results} Notice that all lifespan bounds in Corollaries~\ref{cor:cech-TW}, \ref{cor:cech-AW}, \ref{cor:VR-UW}, \ref{cor:cech-CTW}, and \ref{cor:PCA} are stated in terms of homological lifespans ($d_\omega-b_\omega$).  \Cref{rmk:comparison} implies that all these bounds therefore apply to the length of every interval in  either $\PD_k(\X)$ or $\wc{\PD}_k(\X\subset \Y).$
\end{rmk}

\section{Bounding Lifespans by Spread} \label{sec:spread}

This section is concerned with global upper bounds for \v{C}ech lifespans of homology classes in arbitrary degrees dimension. In terms of the conceptual description laid out on page \pageref{pg:cores}, to formulate these results we will choose cores to be \"ubercontractible spaces, as defined below.

\medskip
Our results in this section are somewhat analogous to the following result concerning VR lifespans.

\begin{thm} [{VR-lifespans via Spread; \cite[Proposition 9.19 and Remark 9.18]{lim2020vietoris}}]
\label{thm:spread}
Let $\X$ be a compact metric space and let $\omega\in\mathrm{Spec}_k(\V_\bullet(\X))$. Then  $d_\omega-b_\omega\leq \spr(\X)$.
\end{thm}

Here Katz's notion of \emph{spread} is used \cite{katz1983filling}; see also \cite[Definition 4]{wilhelm1992filling}. By definition, $\spr(\X)$ is the infimum of the real numbers $\delta \geq 0$ for which there is a finite subset $A = \{a_1, \ldots, a_m\} \subset \X$ with $\diam(A)\leq \delta$ and $\dd^\X_\h(\X, A) \leq \delta$. 

In the rest of this section, we work in a Banach space $\Y$ and, for each compact set $\X \subset \Y$, we consider a certain variant of the notion spread for which we prove an analogue of \Cref{thm:spread}. 

\begin{defn}[\"Ubercontractibility]\label{def:ubercontractible}
A set $\T$, sitting in a Banach space $\Y$, will be said to be \emph{$\delta$-\"ubercontractible} if its neighborhood $\N_r(\T \subset \Y)$ is contractible for every $r \ge \delta$. 
\end{defn}

\begin{rmk} [Acyclicity]
For our purposes  a slightly weaker condition would suffice: all neighborhoods $\N_r(\T\subset \Y)$, $r\ge \delta$, are acyclic (have trivial homology groups) rather than contractible (have trivial homotopy groups). The acyclicity condition is equivalent to having almost trivial \v{C}ech persistence diagram (trivial beyond death time $\delta$), and is implied by \"ubercontractibility as defined above.
\end{rmk}

\begin{defn}[\"Uberspread]\label{def:uberspread}
For a compact subset $\X$ of a Banach space $\Y$, we define its 
\emph{\"uberspread} as 
 $$\uspr(\X\subset \Y):=\inf\left\{\delta \ge 0~\middle\vert~ 
 \begin{array}{l}
 \exists \T\subset \Y \text{ a $\delta$-\"ubercontractible simplicial complex}\\
 \text{such that 
 }\dd_\h^\Y(\X,\T)\leq\delta
 \end{array}\right\}.$$
\end{defn}

Note that for any point $p\in \Y$ the singleton set $\{p\}$ is $\delta$-\"ubercontractible  for all $\delta\geq 0$. In particular, this implies that $$\uspr(\X\subset \Y)\leq \rad(\X\subset \Y).$$ For additional  constructive examples of \"ubercontractible sets, we refer the reader to \Cref{sec:app-ubercontractible}, where we show that the cut-locus of a convex set is typically $0$-\"ubercontractible.

\thmuberspread

\begin{figure}
 \begin{center}
 \includegraphics[width=0.35\linewidth]{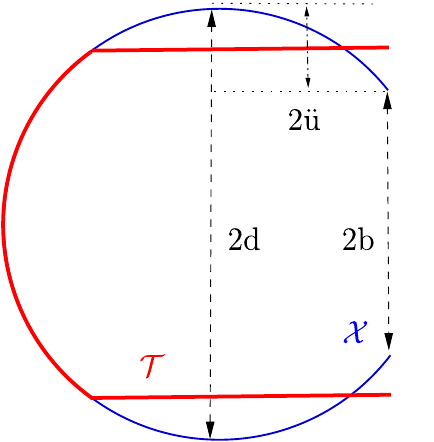}
 \caption{The bound $d-b\leq 2\,\ddot{u}$ in~\Cref{thm:uberspread} is sharp; here $\Y=\mathbb{R}^2$ and $\ddot{u} = \uspr(\X\subset \mathbb{R}^2)$.  \label{fig:uber}}
 \end{center}
\end{figure}

\begin{proof}  Via the functorial nerve theorem (\Cref{thm:cech-neigh}), we will argue at the level of neighborhood filtrations, as opposed to simplicial filtrations.
Let $\T$ be a $\delta$-\"ubercontractible simplicial complex in $\Y$ which is Hausdorff distance at most $\delta$ away from $\X$, where $\delta$ is just a tiny bit greater than $\uspr(\X\subset \Y)$. 
Then, for all $r>0$ we have
$$\begin{tikzcd}
\N_{r}(\X\subset \Y) \arrow[rd,"p_{r}"'] \arrow[rr, "i_{r,r+2\delta}"] & & \N_{r+2\delta}(\X\subset \Y) \\
& \N_{r+\delta}(\T\subset \Y) \arrow[ru,"q_{r+\delta}"']&
\end{tikzcd}$$
where $i_{r,r+2\delta}$, $p_r$, and $q_{r+\delta}$ are the obvious inclusion maps. Applying the homology functor to the above diagram gives, for all $r>0$, that $$\Hom_k(i_{r,r+2\delta}) = \Hom_k(p_{r}) \circ \Hom_k(q_{r+\delta}) = 0$$ since $\N_{r+\delta}(\T\subset \Y)$ is contractible. The this means that $d_\omega-b_\omega\leq 2\delta$. The proof follows.
\end{proof}

\begin{rmk} [Katz spread vs. $\uspr$]
Let us compare \Cref{thm:uberspread} with~\cite[Proposition~9.19]{lim2020vietoris}. Suppose we are given a compact metric space $\X$ and a finite subset $A = \{a_1, \ldots, a_m\} \subset \X$ with $\diam(A) \leq \delta$ and $\dd^\X_\h(\X, A) \leq \delta$. In other words, the \emph{Katz spread} of $\X$ does not exceed $\delta$:
$$\spr(\X) \le \delta.$$
Embed $\X$ in $L^\infty(\X)$ via the Kuratowski map $\kappa$, and note that the set $\kappa(A) \subset L^\infty(\X)$ is $\delta/2$-\"ubercontractible. Hence, we can apply \Cref{thm:uberspread} and conclude that \v{C}ech lifespans of $\kappa(\X) \subset L^\infty(\X)$ do not exceed $2\delta$. Therefore, VR lifespans of $\X$ do not exceed $4\delta$. This is weaker than the conclusion of \cite[Proposition~9.19]{lim2020vietoris}, which tells us that in this situation VR lifespans of $\X$ do not exceed $\spr(\X) \le \delta$.
\end{rmk}

In general, the estimate in \Cref{thm:uberspread} cannot be improved (see \Cref{fig:uber}). However, the factor of $2$ in that estimate can sometimes be removed if the comparison set $\T$ (i.e. the ``core") admits a $1$-Lipschitz nearest-point projection from $\Y$. 

\begin{thm}\label{thm:convex-spread}
Let $\X$ be a compact subset of a Banach space $\Y$. Assume that $\conv(\X)$ admits a $1$-Lipschitz nearest-point projection to a closed convex set $\T \subset \Y$.\footnote{For example, this is true whenever $\Y$ is Euclidean.} Then any $\omega \in \mathrm{Spec}_k(\C_\bullet(\X\subset \Y))$, $k\ge 0$, has lifespan at most $\dd_\h^\Y(\X, \T)$. 
\end{thm}

\begin{rmk}\label{rmk:stab-does-not-apply}
The above theorem \emph{cannot} be derived from the stability results in \Cref{lem:stability1} and \Cref{lem:stability2}, as those results apply \emph{exclusively} to homology classes associated with points in the persistence diagram $\wc{\PD}_k(\X\subset \Y)$. In contrast, \Cref{thm:convex-spread} applies to \emph{all elements of $\mathrm{Spec}_k(\C_\bullet(\X\subset \Y))$}, which, as expressed by \Cref{rmk:not-all-omegas-in-spec} and \Cref{prop:relation},  generally includes strictly more homology classes than those represented by points in the persistence diagram $\wc{\PD}_k(\X\subset \Y)$. Nevertheless, the proof of \Cref{thm:convex-spread} shares ideas with (certain simplicial-level)  proofs of the stability results.
\end{rmk}

\begin{proof}[Proof of \Cref{thm:convex-spread}]
Let $\delta = \dd_\h^\Y(\X, \T)$.
By assumption, there exists a $1$-Lipschitz nearest-point projection $\pi:\conv(\X)\rightarrow \T$; there is also a (possibly discontinuous) map $j:\T\to \X$ such that 
$$\|x - \pi(x)\|\leq \delta\,\,\mbox{and}\,\,\|a-j(a)\|\leq \delta\,\,\mbox{$\forall x\in \X$ and $a\in \T$}.$$

It is easy to see that $\pi$ and $j$ induce, for each $\epsilon>0$, simplicial maps
$$\pi_\epsilon:\C_{\epsilon}(\X\subset \Y)\rightarrow \C_{\epsilon}(\T\subset \Y)\,\,\mbox{and}\,\, j_\epsilon:\C_{\epsilon}(\T\subset \Y)\rightarrow \C_{\epsilon+\delta}(\X\subset \Y).$$ 

To see this in the case of $\pi$ we proceed as follows. Let $\sigma = \{x_0,\ldots,x_n\}\in \C_{\epsilon}(\X\subset \Y).$ Let $y_\sigma\in \conv(\X)$ satisfy $\|x_i-y_\sigma\|\leq \epsilon$ for all $i$. Then, since $\pi:\conv(\X)\to \T$ is 1-Lipschitz, we have
$$\|\pi(x_i) - \pi(y_\sigma)\|\leq \|x_i-y_\sigma\|\leq \epsilon$$
for all $i$. Hence, $\pi_\epsilon(\sigma)\in \C_{\epsilon}(\T\subset \Y).$ 

In the case of $j$, let $\sigma = \{x_0,\ldots,x_n\}\in \C_{\epsilon}(\T\subset \Y).$ Let $y_\sigma\in \Y$ satisfy $\|x_i-y_\sigma\|\leq \epsilon$ for all $i$. Then, 
$$\|j(x_i) - y_\sigma\|\leq \|j(x_i)-x_i\| + \|x_i-y_\sigma\|\leq \delta+\epsilon$$
for all $i$. Hence, $j_\epsilon(\sigma)\in \C_{\epsilon+\delta}(\X\subset \Y).$ 

We then have the following (not-necessarily commutative) diagram:
$$
\begin{tikzcd}
\C_{\epsilon}(\X\subset \Y) \arrow[rd,"\pi_{\epsilon}"'] \arrow[rr, "i_{\epsilon,\epsilon+\delta}"] & & \C_{\epsilon+\delta}(\X\subset \Y) \\
& \C_{\epsilon}(\T\subset \Y) \arrow[ru,"j_{\epsilon}"']&
\end{tikzcd}
$$ where $i_{\epsilon,\epsilon+\delta}$ is the obvious inclusion map.

\medskip
\noindent {\bf Claim:} For each $\epsilon>0$ the maps $j_{\epsilon}\circ \pi_{\epsilon}$ and $i_{\epsilon,\epsilon+\delta}$ are contiguous.

\medskip
To see the claim, let $\sigma = \{x_0,\ldots,x_n\}$ be any simplex in $\C_\epsilon(\X\subset \Y)$. We will prove that 
$\tau:=\sigma\cup j\circ \pi(\sigma)$ is a simplex in $\C_{\epsilon+\delta}(\X\subset \Y).$
This requires us to find a point $y_\tau \in \Y$ such that $\|v - y_\tau\|\leq \epsilon+\delta$ for all $v\in \tau$.

Let $y_\sigma\in \R^d$ be such that $\|x_i - y_\sigma\|\leq \epsilon$ for all $i$. Notice that then we can write $$\tau = \{x_0,\ldots,x_n\}\cup\{j\circ\pi(x_0),\ldots,j\circ\pi(x_n)\}.$$ Then, we let $y_\tau:=\pi(y_\sigma)$ and calculate
$$\max_{v\in \tau}\|v - y_\tau\|=\max\left(\max_i \|x_i -\pi(y_\sigma)\|,\max_i \|j\circ\pi(x_i) - \pi(y_\sigma)\|\right).$$
For the first argument of the maximum above,
\begin{align*}
\|x_i -\pi(y_\sigma)\| &\le \|x_i - \pi(x_i)\| + \|\pi(x_i) - \pi(y_\sigma)\| \\
&\le \|x_i - \pi(x_i)\| + \|x_i - y_\sigma\| \\
&\le \delta + \epsilon.
\end{align*}
For the second argument,
\begin{align*}
 \|j\circ\pi(x_i) - \pi(y_\sigma)\| &\le \|j\circ\pi(x_i)-\pi(x_i)\| + \|\pi(x_i)-\pi(y_\sigma)\| \\
 &\leq \|j\circ\pi(x_i)-\pi(x_i)\| + \|x_i-x_\sigma\| \\ 
 &\leq \delta + \epsilon,
\end{align*}
which establishes the claim.

\medskip
Now, going back to the proof of the theorem: apply the homology functor $\Hom_k$ to the diagram above and obtain the following \emph{commutative} diagram:
$$
\begin{tikzcd}
\Hom_k(\C_{\epsilon}(\X\subset \Y)) \arrow[rd,"\Hom_k(\pi_{\epsilon})"'] \arrow[rr, "\Hom_k(i_{\epsilon,\epsilon+\delta})"] & & \Hom_k(\C_{\epsilon+\delta}(\X\subset \Y)) \\
& \Hom_k(\C_{\epsilon}(\T\subset \Y)) \arrow[ru,"\Hom_k(j_{\epsilon})"']&
\end{tikzcd}
$$
Now, $\Hom_k(\C_{\epsilon}(\T\subset \Y))=0$ since $\T$ is convex. Then,
$$\Hom_k(i_{\epsilon,\epsilon+\delta}) = \Hom_k(j_{\epsilon}) \circ \Hom_k(\pi_{\epsilon}) = 0,$$
which completes the proof.
\end{proof}

\section{Bounding Extinction Times} \label{sec:extinction}

In \Cref{sec:width}, we discussed bounds for the lifespans of individual homology classes in degree $\geq k$ via different notions of $k$-width from metric geometry. However, as \Cref{fig:lowerdim} suggests, these $k$-widths do not give any bound for the lower  homology classes in lower degrees. In other words, if $\AW_{k}(\M \subset \R^N)$ (or another $k$-width) is small, the filling radius (or lifespan) of a homology class $\alpha\in \Hom_j(\M)$ for $j<k$ can still be very large. It is easy to see that the example in \Cref{fig:lowerdim} can be generalized to any dimension and codimension. In particular, $k$-widths do not say much about the size of topological features in dimensions lower than $k$.

In this section, we use a different approach and propose a global bound to the lifespans of homology classes in all degrees at once.  Our goal is to bound the following quantity.

\begin{defn}[Extinction Time] \label{defn:VRextinction} Let $\X$ be a compact metric space.  Then, $$\xi(\X):=\sup_{k\geq 0} \sup \big\{d_\omega \mid \omega \in \mathrm{Spec}_k(\V_\bullet(\X))\big\}$$
 is called the \textit{VR-extinction time} of $\X$. We similarly, define $\wc{\xi}(\X \subset \Y)$,  the \v{C}ech-extinction time for the \v{C}ech filtration  of a subset $\X$ of a Banach space $\Y$. We will use the terms VR-extinction (respectively \v{C}ech-extinction) for short.
\end{defn}

Notice that via the VR-extinction $\xi(\X)$, we are not only bounding the lifespans of homological features with birth time $=0$, but we are bounding the death time of \emph{all} homological features (in degree $1$ and higher) appearing throughout the filtration. Also, note that, as pointed out in \cite[Remark 9.8]{lim2020vietoris}, the radius of $\X$ (\Cref{def:radius}) automatically gives an upper bound for the VR-extinction time, that is, $\xi(\X)\leq \rad(\X)$. This is easy to see since the VR complex $\VR_r(\X)$ is a simplicial cone for $r\geq \rad(\X)$.  A similar statement is clearly true for \v{C}ech-extinction,  i.e. $\wc{\xi}(\X\subset \Y)\leq \rad(\X\subset \Y)$. See~\Cref{rmk:extinction} for comments on the usefulness of this type of bounds.

\begin{rmk}[Utility of Extinction Time Bounds]\label{rmk:extinction} 
Bounds on extinction times can be useful in practice for reducing the computational effort incurred by algorithms designed for calculating PH, as we now explain. By \Cref{prop:relation}, any bound on the extinction times of a compact metric space provides an upper bound on the right endpoints of every interval in the barcode of the VR-filtration for that space; see \Cref{rmk:comparison}. This is particularly useful due to the inherent structure of algorithmic procedures for computing persistence diagrams of simplicial filtrations whose complexity increases with the total number of simplices; see \cite[Chapter VII]{edelsbrunner2010computational}. For instance, both the software packages Ripser~\cite{bauer2021ripser} and Eirene~\cite{henselman2016matroid} use $\rad(X)$ as a cut-off value for the filtration parameter. Any computationally feasible approximation to the extinction time bounds below could similarly boost efficiency in practical applications.
\end{rmk}

In the following, we aim to give much finer estimates for extinction times in both the VR and \v{C}ech settings.

 \begin{rmk} [Motivating Example] \label{rmk:ellipsoid-handle} Here, we give a toy example to motivate the  notion of extinction defined above. 
 Let $\E$ be the $(N-1)$-dimensional ellipsoid in $\R^N$ given by 
 \[
 \E:=\left\{\mathbf{x}\in \R^N ~\middle\vert~ \sum_{i=1}^{N}\dfrac{x_i^2}{a_i^2}=1\right\},
 \] 
 where $a_1>a_2>\dots>a_{N}>0$. While $\E$ has trivial homology groups in low degrees, one can easily add some topology to $\E$ by adding $k$-handles for $0<k<N-1$ as follows. Fix $0<k<N-1$. Let 
 \[
 \overline{\E}:=\left\{\mathbf{x}\in \R^N ~\middle\vert~ \sum_{i=1}^{N}\dfrac{x_i^2}{a_i^2}\leq 1\right\},
 \]
 be the solid ellipsoid. Given an integer $m\geq 1$, we will use to notation $\overline{\E}^{m}:=\{\mathbf{x}\in \overline{\E}\mid x_i=0 \mbox{ for } i>m\}$. Let $\{D_1, D_2,\dots, D_{\ell_k}\}$ be $\ell_k$ disjoint small disks in $\overline{\E}^{N-k}$. Then, consider the following surgery operation. For $1\leq j\leq \ell_k$, let $\Omega_j:=\overline{\E}\cap (D_j\times \R^k)$. Let $S_j:=\E\cap \Omega_j$, and $T_j:=\closure(\partial\Omega_j \setminus S_j)$. Then, swapping the $S_j$ and the $T_j$ will give a new closed manifold 
 \[
 \wh{\E}:= \left(\E\setminus\bigcup_{j=1}^{\ell_k} S_j\right)\cup\bigcup_{j=1}^{\ell_k} T_j \subset \mathbb{R}^N.
 \] 
 See Figure \ref{fig:ehat} for an illustration. While $\Hom_k(\E)$ is trivial for $0<k<N-1$, $\mbox{rank}(\Hom_k(\wh{\E}))\geq {\ell_k}$ because of the homology classes  generated by the $k$-handles $\{T_j\}$. By choosing $\{D_j\}$ such that $\{\Omega_j\}$ are all pairwise disjoint,  one can obtain $\wh{\E}$ with nontrivial homology  in all desired degrees $k$. 

\begin{figure}
\centering
\includegraphics[width = 0.9\linewidth]{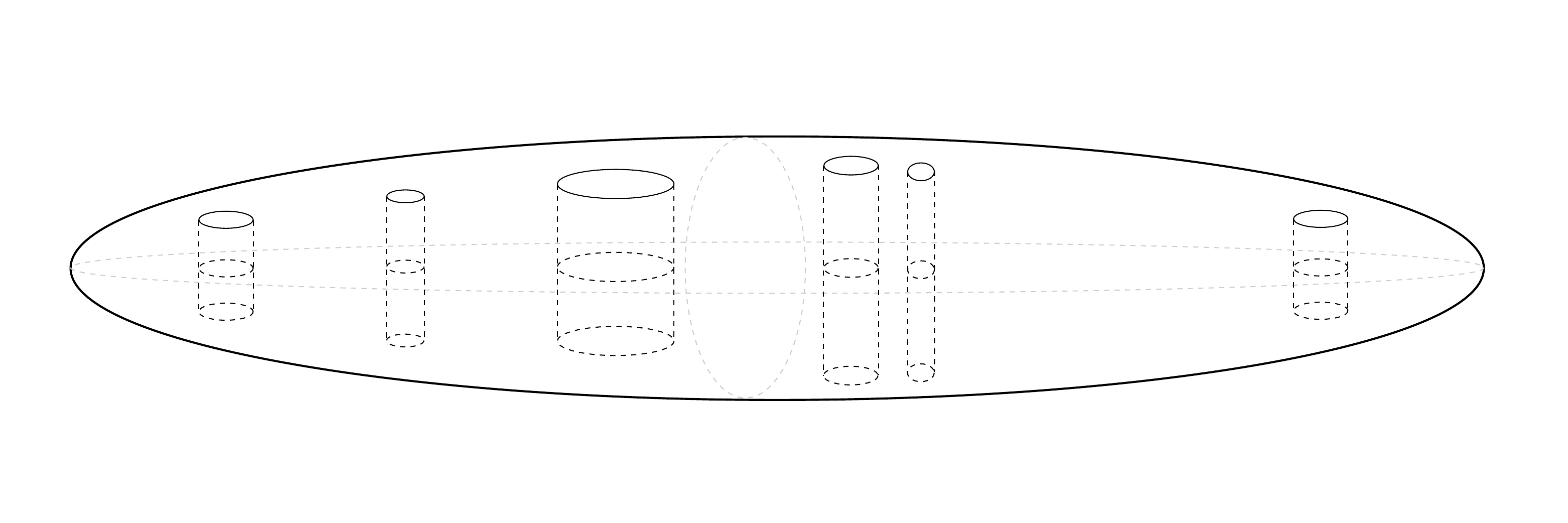}
\caption{The space $\widehat{\mathrm{E}}$ from \Cref{rmk:ellipsoid-handle} for the case $N=3$ and $k=1$.}\label{fig:ehat}
\end{figure}

Notice that, through Kolmogorov widths, \Cref{cor:PCA} yields that  $$d_\omega-b_\omega \leq \KW_k(\wh{\E} \subset \mathbb{R}^N) = a_{k+1}$$
for any class $\omega \in \mathrm{Spec}_k(\C_\bullet(\X\subset \mathrm{R}^N)).$
By suitably choosing the numbers $\{a_i\}$, all $k$-widths for $k\le \dim (\widehat{\mathrm{E}}) -1$ can be made arbitrarily large. However, it is not hard to see that, in this particular example,  the death times of all these homology classes do not exceed the smallest axis length $a_N$. This example therefore shows that $k$-widths can highly overestimate  lifespans. In other words, we have that  the \v{C}ech extinction satisfies $\wc{\xi}(\wh{\E} \subset \mathbb{R}^N)\leq a_N$. In the following sections, we estimate  extinction radii by comparing a given space to a nearby topologically trivial space.

 \end{rmk}

\subsection{Bounding \v{C}ech-extinction via Convex Hulls} \label{sec:Cechextinction}

\begin{defn}[Convexity Deficiency]\label{def:conv-def}
For a compact set $\X$ sitting in a Banach space $(\Y, \|\cdot\|)$, we define the \emph{convexity deficiency} of $\X$ as
$$\cdef(\X \subset \Y) := \dd_\h^\Y(\X, \conv(\X)) = \sup_{y \in \conv(\X)} \inf_{x \in \X} \|y-x\|.$$
\end{defn}

\thmcechextintion

\begin{proof}
For any $r \ge 0$, and any points $x_1, \ldots, x_m \in \X$, the intersection of balls $\bigcap_{i=1}^m B_r(x_i)$ is nonempty in $\Y$ if and only if the intersection $\bigcap_{i=1}^m (B_r(x_i) \cap \conv(\X))$ is non-empty. Therefore, the filtered \v{C}ech complexes $\C(\X\subset \Y)$ and $\C(\X\subset \conv(\X))$ are identical. However, the homotopy type of $\C_r(\X\subset \conv(\X))$ is trivial whenever $r \ge \cdef(\X\subset\Y)$, so no topological feature persists beyond time $\cdef(\X\subset\Y)$. \end{proof}

\begin{rmk} [Stability vs.~Extinction] \label{rmk:stab-ext} The stability theorems (\Cref{sec:stability}) do not imply the previous result for the following reasons:
\begin{enumerate}
    \item In contrast with the stability theorems \Cref{lem:stability1} and \Cref{lem:stability2}, which apply \emph{only} to homology classes associated to points in the persistence diagram $\wc{\PD}_k(\X\subset \Y)$, \Cref{thm:extinct-cdef} applies to \emph{all elements of $\mathrm{Spec}_k(\C_\bullet(\X\subset \Y))$}. Recall, from \Cref{rmk:not-all-omegas-in-spec} and \Cref{prop:relation}, that in general $\mathrm{Spec}_k(\C_\bullet(\X\subset \Y))$ contains strictly more homology classes than those which can be associated with points in the persistence diagram $\wc{\PD}_k(\X\subset \Y)$. 
\item  Strictly speaking, one should not expect that, in general, stability  holds for the lifetime or extinction of arbitrary classes in $\mathrm{Spec}_k(\V_\bullet(\X))$ or $\mathrm{Spec}_k(\C_\bullet(\X\subset \Y))$. For example, the filling radius of an $m$-dimensional manifold $\M$ coincides with both the lifetime and  extinction  of its fundamental class $[\M]\in \mathrm{Spec}_m(\V_\bullet(\M))$. However, the filling radius of $\M$
is \emph{not} stable under the Gromov--Hausdorff distance as explained in \cite[Sections 9.4 and 9.5]{lim2020vietoris}.
\item Even if we restricted ourselves to those homology classes associated with points in the persistence diagram, the stability theorems would only yield bounds on  lifespans but not on death times  (i.e. on extinction). Indeed,  \Cref{lem:stability2} implies that $$\dd_b(\wc{\PD}_k(\X\subset \Y),\wc{\PD}_k(\conv (\X)\subset\Y)))\leq \dd_\h^\Y(\X,\conv(\X)) = \cdef(\X\subset\Y).$$ 
Since $\wc{\PD}_k(\conv(\X)\subset\Y)=\emptyset$, the stability result and the definition of the bottleneck distance  \cite{edelsbrunner2010computational}  imply that for any point $(b,d)\in \wc{\PD}_k(\X\subset \Y)$, $$d-b\leq  2\cdef(\X\subset\Y).$$ However, both $b$ and $d$ can be arbitrarily large without violating this inequality. Meanwhile, \Cref{thm:extinct-cdef} bounds from above the second coordinate of \emph{every} element of $\wc{\PD}_k(\X\subset \Y)$: for any $(b,d)\in \wc{\PD}_k(\X\subset \Y)$, $$d\leq \cdef(\X\subset\Y).$$ 
\end{enumerate}
\end{rmk}

\begin{rmk}
Example in \Cref{fig:uber} shows that it is not true that the extinction $\wc{\xi}(\X \subset \Y)$ is bounded from above by the \"uberspread $\uspr(\X\subset \Y)$.
\end{rmk}

While the above remark shows that the distance to the nearest \"ubercontractible space (\"uberspread) fails to bound extinction, it might be still true that distance to a specific cleverly chosen \"ubercontractible space might give an estimate on extinction. One natural choice for the role of an \"ubercontractible space $\T$ approximating $\X \subset \Y$ is the convex hull of $\X$; this is exactly what has just been discussed; with this choice, there is an extinction bound (Theorem~\ref{thm:extinct-cdef}), but it is tempting to improve it by choosing a finer $\T$. Another natural choice of $\T$ is given by the cut-locus of the boundary of the convex hull of $\X$ (see the discussion in Appendix~\ref{sec:app-ubercontractible}). Unfortunately, the extinction time is not bounded from above by the distance from $\X$ to this $\T$ (chosen as the cut-locus of $\partial \conv(\X)$, assuming $\conv(\X)$ full-dimensional). We omit the discussion of examples in view of the negative nature of the result, but they can be obtained as subsets of a square (one can take a square and cut out a large off-centered disk).

\begin{question} [\v{C}ech Extinction and Cut-Locus of $\conv(\X)$] Let $\X$ be a compact subset of a Banach space $\Y$. Let $\C^{\partial  \conv(\X)}$ be the cut-locus of the convex hull of $\X$ (see \Cref{sec:app-ubercontractible}). 
Is it possible to upper-bound the extinction $\wc{\xi}(\X \subset \Y)$ in terms of the Hausdorff distance to $\C^{\partial  \conv(\X)}$? That is,  is it true that  $$\wc{\xi}(\X \subset \Y)\le C_\Y\cdot \dd_\h^\Y(\X, \C^{\partial  \conv(\X)})$$
for some $C_\Y>0$?

\end{question}

\subsection{Bounding VR-extinction via the Tight Span} \label{sec:VRextinction}  
To derive a bound on VR-extinction times, Theorem~\ref{thm:extinct-cdef} can be applied to the Kuratowski (distance-preserving) embedding of $\X$ into $L^\infty(\X)$. However, this approach yields a sub-optimal result. Here, we present a refinement.

The following definition is analogous to Definition \ref{def:conv-def} in that the convex hull of $\X \subset \Y$ is supplanted by $\ts(\X)$ (see~\Cref{defn:tightspan}). Recall that $\X$ naturally embeds in its tight span, so we can assume that $\X \subset \ts(\X)$.
\begin{defn}[Hyperconvexity Deficiency] \label{defn:hypdef}
The \emph{hyperconvexity deficiency} of a compact metric space $\X$ is defined as the number
$$\hcdef(\X):=\dd_\h^{\ts(\X)}(\X,\ts(\X)) = \sup_{f\in \ts(\X)}\inf_{x\in \X}\|f-\dd_\X(x,\cdot)\|_\infty.$$
\end{defn}

The following corollary to \Cref{thm:vr-neigh} is analogous to Theorem \ref{thm:extinct-cdef}. 
\begin{cor}[Bounding VR Extinction]\label{coro:abs-cont} Let $\X$ be a compact metric space. Then, $$\xi(\X)\leq 2\,\hcdef(\X).$$

Furthermore, this bound is tight (see Remark \ref{rem:tight} below).
\end{cor}

\begin{proof}
The claim follows from  \Cref{thm:vr-neigh} and \Cref{rem:inj} together with the facts that (1) $\N_t(\X \subset \ts(\X)) = \ts(\X)$ for all $t\geq \hcdef(\X)$ and (2) $\ts(\X)$ is contractible.
\end{proof}

\begin{rmk}[Tightness of the Bound]\label{rem:tight}
Let $\X$ be the unit $\ell^\infty$ sphere in $\R^2$. Then, in that case, by results of K{\i}l{\i}{\c{c}} and Ko{\c{c}}ak~\cite{kilicc2016tight}, $\ts(\X)$ is isometric to $\Big([-1,1]\times[-1,1],\ell^\infty\Big)$ and we compute that $\hcdef(\X) = 1$. Since $\N_t(\X\subset \ts(\X))\simeq \X\simeq S^1$ for \emph{every} $0<t<1$, and for $t\geq 1$ we have $\N_t(\X\subset \ts(\X)) = \ts(\X)$, which is contractible, by  \Cref{thm:vr-neigh} we have 
\begin{itemize}
\item $\V_r(\X)\simeq S^1$ for every $0<r < 2$, and
\item $\V_r(\X)$ is contractible for $r\geq 2$.
\end{itemize}
Hence,  for $\omega=[S^1]$, $d_\omega = 2$; see also \cite[Corollary 7.13]{lim2020vietoris}. 
\end{rmk}

Recall from \cite[Remark 9.8]{lim2020vietoris} that $\VR_r(\X)$ becomes contractible as soon as $r\geq \rad(\X)$. The next proposition proves that the bound in Corollary \ref{coro:abs-cont} is never worse than this bound. 
\begin{prop}\label{prop:rad-abs-cont} The inequality 
$$2\hcdef(\X)\leq \rad(\X)$$
holds for every compact metric space $\X$.
\end{prop}

\begin{rmk}[Comparison of $\rad(\X)$ and $\hcdef(\X)$]
The upper bound given by Corollary \ref{coro:abs-cont} can be much smaller than $\rad(\X)$.
Indeed, Let $\X$ be any \emph{metric tree}, then, in that case (by item (3) of Proposition \ref{prop:ts}) $\ts(\X)=\X$ so that $\hcdef(\X)=0$. However, $\X$ can be chosen so that $\rad(\X)$ (and also its spread) are arbitrarily large. 
\end{rmk}

\begin{proof}[Proof of Proposition \ref{prop:rad-abs-cont}]
 Assume that $\delta>\rad(\X)$
 and let $x_0\in \X$ be a point such that $\dd_\X(x_0,x)<\delta$ for all $x\in \X$. Pick any $f\in \ts(\X)$ and recall that, according to equation (\ref{eq:char}), we have that 
$$f(x_0) = \max_{x'\in \X}\big(\dd_\X(x_0,x')-f(x')\big).$$

Let $x'_0 \in \X$ be such that $f(x_0) = \dd_\X(x_0,x'_0) - f(x_0')$.
Notice that then, by equation (\ref{eq:char}), we have both 
$$\|f-\dd_\X(x_0,\cdot)\|_\infty = f(x_0)\,\,\mbox{and}\,\,\|f-\dd_\X(x_0',\cdot)\|_\infty = f(x_0').$$
Adding these two expressions together we obtain that 
$$\|f-\dd_\X(x_0,\cdot)\|_\infty + \|f-\dd_\X(x_0',\cdot)\|_\infty = f(x_0) + f(x_0') = \dd_\X(x_0,x_0') <\delta.$$
From this, we conclude that 
$$\inf_{x\in \X}\|f-\dd_\X(x,\cdot)\|_\infty\leq \min\big(\|f-\dd_\X(x_0,\cdot)\|_\infty, \|f-\dd_\X(x_0',\cdot)\|_\infty\big) < \frac{\delta}{2}.$$
Since $f\in \ts(\X)$ was arbitrary, this proves that $\hcdef(\X) <  \frac{\delta}{2}$ from which the claim follows.
 \end{proof}

\section{Final Remarks}\label{sec:remarks}

Here we provide some remarks that could suggest further exploration.

\paragraph*{Widths and Lifespans.} Although our results providing bounds on lifespans via widths are primarily theoretical, they offer practical value both in terms of improving the interpretability of PH features and in applications. Computing the exact Alexandrov, Urysohn, or Kolmogorov 
$k$-width for a given set 
$\X$ is often computationally challenging. However, by definition, these widths arise as infima of certain measurements over 
$k$-dimensional spaces (the ``cores"). While identifying the optimal 
$k$-dimensional space could be highly complex, any meaningful and well-chosen 
$k$-dimensional space can yield relevant measurements that serve as upper bounds for these widths and, consequently, for the lifespans. Thus, even if calculating the optimal bound is infeasible, our results can be effectively leveraged to provide rough yet meaningful upper bounds for the lifespans of significant topological features. See \Cref{rmk:extinction} for other considerations related to  potential uses of our bounds.

We highlight that our framework enables a bidirectional exchange of concepts, integrating ideas from metric geometry into applied algebraic topology and supporting their application in the reverse direction. 
Specifically, since  lifespans are bounded above by various notions of width introduced in \Cref{sec:width}, they consequently provide lower bounds for these quantities.
In other words, for a given metric space $\X$, the maximum lifespan over classes in $\mathrm{Spec}_k(\VR_\bullet(\X))$ (with zero birth time) serves as a lower bound for the corresponding width (e.g., $\UW_{k-1}(\X)$). 
This relationship offers a practical  approach for estimating widths by leveraging topological persistence. We exemplify this now.

\begin{ex}[The Urysohn Width of the $n$-Torus]\label{ex:uw-ph} Let $n\geq 1$ be any integer and $a_1\geq a_2\geq \cdots \geq a_n>0$. Consider the $n$-torus $T^n := a_1 S^1 \times a_2 S^1\times \cdots \times a_n S^1$ endowed with the $\ell^\infty$ product metric, and where each $a_i S^1$ factor has the geodesic metric (with diameter $\pi a_i$). Then, we claim that for any $k\in\{1,\ldots,n\}$ we have
$$\frac{2\pi}{3}a_k\leq \UW_{k-1}(T^n)\leq \pi a_k.$$
The upper bound  can be trivially obtained by considering the projection onto the first $k-1$ factors of $T^n$.  The lower bound can be obtained through an  argument via PH as follows:
\begin{enumerate}
    \item[(1)] As proved by  Adamaszek and Adams in~\cite{adamaszek2017vietoris}, $I:=\left(0,\tfrac{2\pi}{3}\right]$ is the only bar in the VR-barcode of  $S^1$ whose left endpoint is zero. Furthermore, this interval appears in degree-1. 
    \item[(2)] By the K\"unneth formula for the VR-barcodes of  $\ell^\infty$-products of compact metric spaces (see e.g. \cite[Theorem 6.1 and Example 6.4]{lim2020vietoris}), the only intervals in the degree-$k$ VR-barcode of $T^n$ with zero left endpoint must arise from intersecting exactly $k$ of the intervals $a_1I,a_2I,\cdots, a_nI$ each corresponding to the degree-1 VR-barcode of one of the $n$ $S^1$-factors of $T^n$. These intersections are precisely of the form
 $(0,d(L)]$ where $L\subset \{1,\ldots,n\}$ s.t.  $|L|=k$ and
    $$d(L):=\tfrac{2\pi}{3}\min_{\ell\in L} a_\ell.$$
   Then, the maximum  of $d(L)$ over all such subsets $L$ equals $\tfrac{2\pi}{3}a_k$.
   \item[(3)] By \Cref{thm:fillrad-urysohn-width} and \Cref{prop:fillrad} we now conclude that $\tfrac{2\pi}{3}a_k\leq \UW_{k-1}(T^n).$
\end{enumerate}
Compare with \cite[(E$_1$)]{gromov1988width} and see also \cite[page 8]{gromov1983filling}. 
\end{ex}

\paragraph*{Standard PCA vs. Lifespans.} In \Cref{sec:PCA}, we give bounds for \v{C}ech lifespans via the variances $\{\nu_k\}$  induced by PCA$_\infty$ (\Cref{cor:PCA}) As one can easily notice, in order to have such a rigorous bound, we modified the usual PCA definition and considered an $\ell^\infty$-variant. However, this makes the PCA$_\infty$ structure highly sensitive to  outliers in comparison with the original PCA. On the other hand, standard PCA is a mainstream, highly effective dimension reduction tool for real-life applications with several very efficient computational techniques available. While our results do not say anything about the relation between the original PCA and PH, an experimental result relating the lifespans of bars in the persistence diagram of the \v{C}ech filtration $\C_\bullet(\X\subset \R^N)$
 with variances $\wt{\nu}_{k+1}(\X)$ of original PCA would be very interesting, for a given finite set of points $\X\subset \R^N$. It would be particularly interesting  and useful for real-life applications to carry out a statistical comparative analysis for random finite subsets $\X$ in $\R^N$ (and for $k\geq 1$).

\paragraph*{Principal Curves and Surfaces.}  
One can notice that when describing the width-based  arguments in \Cref{sec:width}, we first introduce a $k$-dimensional optimal core $\Lambda_k$ for a given set $\X$, then the $k$-width $\mathrm{W}_k(\X)$ is defined as some kind of ``distance'' from $\X$ to $k$-core $\Lambda_k$. Hence, when $\mathrm{W}_k(\X)$ is small, in metric geometry, $\X$ is regarded as ``essentially $k$-dimensional''. There is a similar notion in statistics called {\em Principal Curves and Surfaces}; see Hastie and Stuetzle~\cite{hastie1989principal}, Delicado~\cite{delicado2001another}, and Ozertem and Erdogmus~\cite{ozertem2011locally}. While principal curves and surfaces are defined as $1$- and $2$-dimensional objects, one can easily generalize the idea to any dimension $k$, e.g. principal $k$-manifolds. In our setting, for a compact subset $\X$ in a Banach space $\Y$, principal curves and surfaces can be considered as $k$-dimensional objects $\Sigma_k$ which minimize the $\ell^2$-distance from $\X$ to $\Sigma_k$ for $k=1,2$ with some normalization condition on $\Sigma_k$. In this sense, principal curves and surfaces can be viewed as nonlinear generalizations of principal component analysis (PCA).

Similarly, for a given set $\X$, our $k$-cores and principal $k$-manifolds can be regarded as analogous constructs which extend the underlying idea to a broader geometric framework. In \Cref{sec:PCA}, we defined PCA$_\infty$ as $\ell^\infty$-version of the original $\ell^2$ PCA . Similarly, our $k$-cores $\Lambda_k$ minimize the $\ell^\infty$-distance between $\X$ and $\Lambda_k$ while principal $k$-manifolds $\Sigma_k$ minimize the $\ell^2$-distance between $\X$ and $\Sigma_k$. By using this analogy, as principal curves and surfaces are suggested as {\em dimension reduction} method, one can consider our $k$-cores represent the essential structure of $\X$ when $\mathrm{W}_k(\X)$ is small. Furthermore, just like the discussion in the preceding paragraph  (PCA vs PCA$_\infty$), it would be interesting to carry out a statistical study of the relationship between the \v{C}ech lifespans and $\ell^2$-distance to principal $k$-manifolds.

\appendix

\section{Auxiliary Properties of Widths}\label{sec:app-AW-UW}
\setcounter{thm}{0}

It was mentioned in \Cref{sec:width-background} that
$$\AW_k(\X\subset\Y) \leq \UW_k(\X) \leq 2\AW_k(\X\subset\Y).$$

The right-hand side inequality is trivial, whereas the inequality on the left requires an explanation.

\begin{lem}[Alexandrov~\cite{alexandroff1933urysohnschen}]
\label{lem:width-width}
For any compact set $\X$ in a Banach space $\Y$,
\[
\AW_k(\X\subset \Y) \leq \UW_k(\X).
\]
\end{lem}

\begin{proof}
Suppose $\UW_k(\X) < \delta$, and let us show that $\AW_k(\X\subset \Y) < \delta$. There is a continuous map $f: \X \to \Delta^k$ to a finite $k$-dimensional complex with fibers of diameter $< \delta$. Subdivide $\Delta^k$ very finely, so that the preimage of any open star of $\Delta^k$ under the map $f$ has diameter $< \delta$.\footnote{Recall that the open star $S_v$ of a vertex $v \in \Delta^k$ is the union of the relative interiors of all simplices of $\Delta^k$ that contain $v$.} For each vertex $v \in \Delta^k$, pick a point $c_v$ in the preimage of the open star $S_v$ (unless this preimage is empty; in this case, we can safely remove $S_v$ from $\Delta^k$). The ball $B_\delta(c_v) \subset \Y$ of radius $\delta$ centered at $c_v$ covers $f^{-1}(S_v)$. Consider an auxiliary map $\gamma: \Delta^k \to \Y$ defined by sending any vertex $v \in \Delta^k$ to $c_v \in \Y$, and then extending linearly on $\Delta^k$. Compose this map with $f$, and consider $\gamma \circ f: \X \to \Y$. Its image is a simplicial complex of dimension at most $k$, and to complete the proof it suffices to show that every point $x \in \X$ is moved by distance $\|x - \gamma(f(x))\| < \delta$. Let $f(x)$ lie in the relative interior of a simplex of $\Delta^k$ with the vertices $v_0, \ldots, v_m$. Then $x \in f^{-1}(S_{v_i}) \subset B_\delta(c_{v_i})$, and $\|x- c_{v_i}\| < \delta$, where $0\leq i \leq m$. By construction of $\gamma$, the point $\phi(f(x))$ lies in the convex hull of the points $c_{v_i}$, $0\leq i \leq m$. Therefore, $\|x - \gamma(f(x))\|$ does not exceed the maximum of $\|x - c_{v_i}\|$ over $0\leq i \leq m$, and this maximum is less than $\delta$.
\end{proof}

\begin{rmk}\label{rmk:aw=uw}
Depending on the geometry of $\Y$, this inequality may be slightly improved if for each vertex $f^{-1}(S_v)$ we cover $f^{-1}(S_v)$ by a ball of the smallest possible radius. For example, if $\Y = \R^N$ is Euclidean, then it follows from Jung's theorem~\cite{jung1901ueber} that this radius can be taken to be $\delta \sqrt{\frac{N}{2(N+1)}}$ (instead of $\delta$). The rest of the proof runs without changes, and the final result is $\AW_k(\X \subset \R^N) \leq \sqrt{\frac{N}{2(N+1)}} \UW_k(\X) < \frac{1}{\sqrt{2}} \UW_k(\X)$. Another extreme example is when $\Y$ is hyperconvex (for example, $L^\infty(\X)$), and any bounded set $\A$ can be covered by a ball of radius $\frac12 \diam(\A)$. In this case, our estimates actually imply that $\UW_k(\X) = 2\AW_k(\X\subset\Y)$. 
\end{rmk}

The following property was used in \Cref{sec:treewidth}.

\begin{lem}
\label{lem:treewidth-continuity}
The treewidth enjoys the following continuity property. Let $\X_1 \supset \X_2 \supset \cdots$ be a nested sequence of compact sets in a Banach space $\Y$, and let $k$ be a non-negative integer. Then
\[
\lim_{i \to \infty}\TW_k(\X_i \subset \Y) = \TW_k\left(\bigcap_i \X_i \subset \Y\right).
\]
\end{lem}

\begin{proof}
Denote $\X = \bigcap \X_i$, $w = \TW_k(\X\subset\Y)$, and let $f:\X\to \Y$ be a witness map: a continuous map whose image lies in a finite contractible $k$-dimensional simplicial complex $\Delta \subset Y$, and such that $\|x-f(x)\| < w + \epsilon/3$ for all $x \in \X$ and some arbitrarily chosen $\epsilon > 0$. For each point $y \in \Delta$, the fiber $f^{-1}(y)$ lies in the open ball $B_{w+\epsilon/3}(y)$ (note that $f^{-1}(y)$ can be empty). Since $f$ is continuous on the compact set $\X$, it is uniformly continuous, hence there is a tiny radius $\rho > 0$ such that such that the ``thickened fiber'' $f^{-1}(B_\rho(y))$ lies in the open ball $B_{w+2\epsilon/3}(y)$, for each $y \in \Delta$. Since $\Delta$ is a finite complex of dimension $k$, there is a finite closed cover $\bigcup D_j = \Delta$ such that its multiplicity is $k+1$, its nerve is homotopy equivalent to $\Delta$, each $D_j$ fits in a ball $B_\rho(y_j)$, and $C_j := f^{-1}(D_j)$ fits in the ball $B_{w+2\epsilon/3}(y_j)$. 

We claim that the sets in the closed cover $\bigcup C_j = \X$ can be inflated a little bit while preserving their intersection pattern. For each collection of indices $J = \{j_1, \ldots, j_n\}$ such that the intersection $\bigcap\limits_{j \in J} C_j$ is empty, it follows from the compactness of $\X$ that the number
\[
\delta_J = \min\limits_{x\in \X} \max_{j \in J} \dist(x, C_j)
\]
is attained and positive. Take a positive $\delta$ smaller than $\epsilon/3$ and also smaller than each $\delta_J$ over all collections $J$ such that $\bigcap\limits_{j \in J} C_j = \varnothing$. Consider the open cover $\{U_j\}$ of $\X$, where $U_j = \N_\delta(C_j \subset \Y)$. It has the same nerve as $\{C_i\}$, and each $U_j$ is contained in $B_{w+\epsilon}(y_j)$. We have 
\[
\X \subset \bigcup_j U_j = \N_\delta(\X\subset\Y). 
\]
Now pick a partition of unity $\{\psi_j\}$ subordinate to $\{U_j\}$. Use it to map $\bigcup_j U_j$ to its nerve. Namely, a point $x \in \bigcup_j U_j$ is mapped to $\sum\limits_j \psi_j(x) y_j$. This gives a (possibly non-surjective) map $\bigcup_j U_j \to \Delta$. Since $U_j \subset B_{w+\epsilon}(y_j)$, every point under this map is shifted by distance less than $w+\epsilon$.

Notice that the open set $\bigcup_j U_j$ contains all $\X_i$ for $i$ large enough, as it can be easily deduced from compactness. Therefore, we have a map from $\X_i$, for all $i$ large enough, to a contractible $k$-dimensional simplicial complex inside $\Y$, and every point is shifted by a distance less than $w+\epsilon$. Hence,
\[
\TW_k(\X_i \subset \Y) - \TW_k(\X \subset \Y) < \epsilon.
\]
Since $\epsilon$ was arbitrary, the result follows.
\end{proof}

\section{\"Ubercontractible Sets}\label{sec:app-ubercontractible}
\setcounter{thm}{0}

Here we only work in $\Y = \mathbb{R}^N$ with Euclidean metric, and we are interested in $0$-\"ubercontractible sets, which we call \"ubercontractible. A rich source of \"ubercontractible sets comes from the \emph{cut-locus} construction, usually considered in a more general Riemannian setting; see, e.g.,  Wolter~\cite{wolter1985cut}. A very close concept is that of \emph{skeleta} (see, e.g., Fremlin~\cite{fremlin1997skeletons}), different from our definition only in that we take the closure. In computational geometry, low-dimensional skeleton construction are also called \emph{medial axes}; see, e.g., the survey by Saha, Borgefors and Sanniti di Baja~\cite{saha2016survey}. We only formulate the definition of the cut-locus for convex hypersurfaces in $\Y = \mathbb{R}^N$.

\begin{defn}
 Let $K \subset \Y$ be a convex body (that is, a compact convex set with a nonempty interior). For each $x \in K$, consider the largest closed ball centered at $x$ and contained in $K$; let $r^{\partial K}(x) \ge 0$ be its radius, and $B^{\partial K}(x) \subset \partial K$ be the set of the points where this ball touches the boundary of $K$. The \emph{cut-locus} of $\partial K$ is the closure of the set of all such $x \in K$ for which the cardinality of $B^{\partial K}(x)$ is at least $2$:
 \[
 C^{\partial K} = \closure \big(\{x \in K: |B^{\partial K}(x)|>1\}\big).
 \]
 There is a canonical retraction $c^{\partial K}: K \to C^{\partial K}$ defined as follows. For $x \in C^{\partial K}$, set $c^{\partial K}(x) = x$. For each point $x \in K\setminus C^{\partial K}$, let $b^{\partial K}(x)$ be the only element of $B^{\partial K}(x)$, and consider the largest closed ball contained in $K$ and touching $\partial K$ at $b^{\partial K}(x)$. The center of this ball belongs to the cut-locus and will be denoted $c^{\partial K}(x)$.
\end{defn}

\begin{thm}\label{thm:cutlocus}
 The cut-locus $C^{\partial K}$ of the boundary of a convex polytope $K$ is \"ubercontractible.
\end{thm}

The key lemma in the proof comes from the Morse theory for the distance functional. It was applied by B\'ar\'any, Holmsen and Karasev in~\cite[Section~3]{barany2015topology} to give a sufficient condition for a set to be contractible, but the proof there, in fact, guarantees that the set is \"ubercontractible.

\begin{lem}[{cf.~\cite[Theorem~2]{barany2015topology}}]\label{lem:morse}
Let $\T \subset \Y$ be a union of finitely many compact convex sets. For each $x \notin \T$ consider the largest ball centered at $x$ whose interior does not meet $\T$. Let $B^\T(x)$ be the set of the points where this ball touches $\T$. Suppose that for each $x \notin \T$ we have $x \notin \conv(B^\T(x))$. Then $\T$ is \"ubercontractible.
\end{lem}

\begin{proof}[Proof of \Cref{thm:cutlocus}]
Suppose $K$ is a convex polytope. It can be shown that $C^{\partial K}$ is a polyhedral complex of codimension $1$ in $Y$. Let $F$ be a facet of $K$, and let $n_F$ be the corresponding inner normal vector. The (relative) boundary of $F$ lies in $C^{\partial K}$. The map $c^{\partial K}$, restricted to $F$, send each point $x \in F$ along $n_F$ until it hits $C^{\partial K}$. Together $F$ and the polyhedral surface $c^{\partial K}(F)$ bound a convex polytope $K_F$. 

Consider a point $x \in K \setminus C^{\partial K}$. It then lies in $K_F$, for some facet $F$. It is sufficient to verify the assumption of \Cref{lem:morse}: one needs to check that $x \notin \conv(B^\T(x))$, where $\T = C^{\partial K}$. Indeed, for every $y \in B^\T(x)$, the vector $y-x$ forms an acute angle with $n_F$, and therefore, the entire set $B^\T(x)$ lies in the open halfspace $\{z \in Y: \langle z-x, n_F \rangle > 0\}$. Therefore, $x \notin \conv(B^\T(x))$.
\end{proof}

We conclude this section by speculating how \Cref{thm:cutlocus} can be proven for convex bodies other than polytopes. Unfortunately, the case of a general convex body $K$ cannot be proven by approximating $K$ with polytopes $K_i$, because in order to guarantee $C^0$-convergence $C^{\partial K_i} \to C^{\partial K}$, we cannot get away just with polytopes. We sketch an argument that works for fairly general convex bodies, modulo some technicalities. We take a different approach and assume the following ``tameness'' assumption: $C^{\partial K}$ consists of finitely many compact convex sets. It allows us to apply \Cref{lem:morse}, but it is highly likely that the lemma can also be stated and proven in greater generality. Not only the class of tame convex bodies includes polytopes, but it also seems to be $C^2$-dense among all $C^2$-smooth convex bodies (we do not discuss this in detail since this digresses too far from the main topic of the paper).

\begin{proof}[Sketch of the proof of \Cref{thm:cutlocus} for tame convex bodies]
\hfill\\

The following two properties will be shown to imply the \"ubercontractibility of $C^{\partial K}$.

\begin{enumerate}
 \item The cut-locus of $\partial K$ consists of finitely many compact convex sets (so we can use \Cref{lem:morse}).
 \item The boundary of $K$ is $C^2$-smooth and strongly convex in the sense that its second fundamental form is positive definite.
\end{enumerate}

The tameness assumption guarantees the first property. We argue that we can inflate $K$ slightly to get a new convex body $K'$ that satisfies the second property, while preserving cut-locus $C^{\partial K'} = C^{\partial K} = \T$. This amounts to choosing $r^{\partial K'}(x)$ for $x \in \T$ carefully, so that on every face of $\T$, $r^{\partial K'}$ is $C^1$-close to $r^{\partial K}$, and $r^{\partial K'}$ is strongly convex. The details are omitted.

In the rest of the proof we assume that $K$ satisfies the two properties above, and we will verify the assumption of \Cref{lem:morse} for $\T = C^{\partial K}$ to show that it is \"ubercontractible. Pick points $x \in K \setminus \T$, and $y \in B^\T(x)$. Our goal is to prove that $\langle y-x, n \rangle > 0$, where $n$ is the inner normal to $\partial K$ at $b^{\partial K}(x)$. This will imply that $x \notin \conv(B^\T(x)) \subset \{z \in \Y: \langle z-x, n \rangle > 0\}$.

Suppose, for the contrary, that $\langle y-x, n \rangle \le 0$. Parametrize the straight line segment $[x,y]$ linearly as $x(t), t \in [0,1]$, $x(0) = x$, $x(1) = y$. For each $t \in [0,1)$, let $u(t) = b^{\partial K}(x(t))$ be the only element $B^{\partial K}(x(t))$. Consider also $v(t) = c^{\partial K}(x(t))$, for $t \in [0,1]$, and notice that $v(1) = y$. The key idea is to look at the continuous family of straight line segments $[x(t), v(t)]$, and investigate how they intersect the open ball $O$ centered at $x$ of radius $r^{\partial K}(x)$. We make the following observations.
\begin{itemize}
 \item The open ball $O$ does not intersect $\T$, and $v(t) \in \T$. So for each $t \in [0,1)$, the segment $[x(t), v(t)]$ starts inside $O$ and ends outside of $O$. 
 \item The strong convexity of $\partial K$ implies that the angle between the vectors $y-x$ and $v(t)-x(t)$ is strictly increasing for $t \in [0,1)$. To see that, one needs to differentiate in $t$ the inner normal to $\partial K$ at $u(t)$ (which is collinear with $v(t)-x(t)$).
 \item Initially, the angle between the vectors $y-x$ and $v(0)-x(0)$ is non-acute (since by assumption $\langle y-x, n \rangle \le 0$). Therefore, for $t$ close to $1$ the angle between $y-x$ and $v(t)-x(t)$ is obtuse, and tends to a limit that is obtuse. It follows that the length of the part of $[x(t), v(t)]$ that lies in $O$ is bounded away from $0$. 
 \item But, as $t \to 1$, the segment $[x(t), v(t)]$ degenerates to the point $y$, so its length must approach zero. This contradiction concludes the proof.
\end{itemize}
\end{proof}

\clearpage 

\section{Notation Table}

\begin{table}[ht!]
\centering
\setlength\tabcolsep{5pt}
\caption{Notation and main symbols.\label{notations}}
\resizebox{1.\linewidth}{!}{
\begin{tabular}{lll}
\toprule
\textbf{Notation} & \textbf{Definition} & \textbf{Place}\\
\midrule
$\conv(\X)$ & Convex hull of $\X$ in a Banach space & \\
$\cconv(\X)$ & Closure of convex hull of $\X$ in a Banach space & \\
$\h_k(\cdot)$ & Singular homology in degree $k$ \\
$\Hom_k(\cdot)$ & Reduced singular homology in degree $k$ \\
$\N_r(\X\subset\mathcal{Z})$ & Open $r$-neighborhood of $\X$ in $\mathcal{Z}$ &\Cref{sec:PH} \\
$\overline{\N}_r(\X\subset\mathcal{Z})$ & Closed $r$-neighborhood of $\X$ in $\mathcal{Z}$ & \Cref{sec:PH}\\
$\V_r(\X)$ & Vietoris--Rips complex of $\X$ for distance $r$ & \Cref{defn:VR} \\
$\C_r(\X\subset\mathcal{Z})$ & \v{C}ech complex of $\X\subset\mathcal{Z}$ for distance $r$ & \Cref{defn:Cech} \\
$\mathrm{Spec}_k(\Delta_\bullet)$ & Homological spectrum of the filtration $\Delta_\bullet$  & \Cref{defn:spectrum} \\
$\PD_k(\X)$ & $k\textsuperscript{th}$-persistence diagram of $\X$ with VR filtration & \Cref{sec:PHbackground} \\
$\widecheck{\PD}_k(\X\subset \Y)$ & $k\textsuperscript{th}$-persistence diagram of $\X\subset\Y$ with \v{C}ech filtration & \Cref{sec:PHbackground} \\
$b_\omega,d_\omega$ &  Birth and death time of the homology class $\omega$ & \Cref{def:lifespan} \\
$d_\omega-b_\omega$ & Persistence (lifespan) of the homology class $\omega$ & \Cref{def:lifespan} \\
$\dd_b(\cdot,\cdot)$ & Bottleneck distance between persistence diagrams & \Cref{sec:stability} \\
$\dd_\h^\Y(\cdot,\cdot)$ & Hausdorff distance between two subsets in $\Y$ & \Cref{sec:PH} \\
$\dd_{\mathrm{GH}}(\cdot,\cdot)$ & Gromov--Hausdorff distance between two metric spaces & \Cref{sec:PH} \\
$L^\infty(\X)$ & Space of bounded functions on $\X$ with $\sup$ norm & \Cref{sec:ts} \\
$\ts(\X)$ & Tight span of $\X$ & \Cref{defn:tightspan} \\
$\rho(\M)$ & Gromov's filling radius of $\M$ & \Cref{defn:filling_radius} \\
$\rho(\omega; \X \subset \Y)$ & The relative filling radius of $\omega\in \Hom_k(\X)$ relative to $\Y$ & \Cref{defn:homology_filling_radius} \\
$\rho(\omega; \X)$ & The absolute filling radius of $\omega\in \Hom_k(\X)$ & \Cref{defn:homology_filling_radius} \\
$\rad(\X)$ & The radius of $\X$ & \Cref{def:radius} \\
$\rad(\X \subset \Y)$ & The circumradius of $\X$ relative to $\Y$ & \Cref{def:radius} \\
$\UW_k(\X)$ & Urysohn $k$-width of $\X$ & \Cref{def:Urysohn} \\
$\AW_k(\X \subset \Y)$ & Alexandrov $k$-width of $\X$ relative to $\Y$ & \Cref{defn:AW} \\
$\TW_k(\X\subset\Y)$ & $k\textsuperscript{th}$ treewidth of $\X$ relative to $\Y$ & \Cref{defn:treewidth} \\
 $\TW^C_k(\X\subset\Y)$ & $C$-robust $k$-dimensional treewidth of $\X$ & \Cref{defn:robust-treewidth} \\
$\KW_k(\X\subset \Y)$ & Kolmogorov $k$-width of $\X$ relative to $\Y$ & \Cref{defn:KW} \\
$\nu_k(\X)$ & $k\textsuperscript{th}$ variance of $\X$ (PCA$_\infty$) & \Cref{defn:pca-infty} \\
$\uspr(\X\subset \Y)$ & \"Uberspread of $\X$ relative to $\Y$ & \Cref{def:uberspread} \\
$\xi(\X)$ & VR-extinction time of $\X$ & \Cref{defn:VRextinction} \\
$\wc{\xi}(\X \subset \Y)$ & \v{C}ech-extinction time of $\X$ relative to $\Y$ & \Cref{defn:VRextinction} \\
$\cdef(\X \subset \Y)$ & Convexity deficiency of $\X$ relative to $\Y$ & \Cref{def:conv-def} \\

 $\hcdef(\X)$ & Hyperconvexity deficiency of $\X$ & \Cref{defn:hypdef}\\ 
\bottomrule
\end{tabular}}
\end{table}

\clearpage

\bibliography{references}
 \bibliographystyle{alpha}

\end{document}